\documentclass{amsart}
\usepackage{appendix}
\usepackage{tabularx}
\usepackage[numbers,sort&compress]{natbib}
\usepackage{color}
\usepackage{tikz}
\usetikzlibrary{shapes,arrows}

\tikzstyle{decision} = [diamond, draw, fill=blue!20, 
text width=6em, text badly centered, node distance=3cm, inner sep=0pt]
\tikzstyle{block} = [rectangle, draw, fill=blue!20, 
text width=9em, text centered, rounded corners, minimum height=4em]
\tikzstyle{block2} = [rectangle, draw, fill=yellow!20, 
text width=9em, text centered, rounded corners, minimum height=4em]
\tikzstyle{line} = [draw, -latex']
\tikzstyle{cloud} = [draw, ellipse,fill=red!20, node distance=3cm,
minimum height=4em]

\makeatletter

\newcommand{\Rmnum}[1]{\expandafter\@slowromancap\romannumeral #1@}
\makeatother

\newtheorem{theorem}{Theorem}[section]
\newtheorem{lemma}[theorem]{Lemma}

\theoremstyle{definition}
\newtheorem{definition}[theorem]{Definition}

\theoremstyle{remark}
\newtheorem{remark}[theorem]{Remark}

\numberwithin{equation}{section}
\allowdisplaybreaks[4]

\begin{document}


\tikzstyle{block} = [rectangle, draw, 
text width=5.5em, text centered, rounded corners, minimum height=2em]

\title{long-time existence and uniqueness of the heterotic $G_2$ flow}


\author{Chuanhuan Li}
\address{ Shanghai Institute for Mathematics and Interdisciplinary Sciences (SIMIS), Shanghai 200433, China \newline
${\quad}$ Research Institute of Intelligent Complex Systems, Fudan University, Shanghai 200433, China}
\curraddr{}
\email{chli@simis.cn}
\thanks{}

\author{Yi Li}
\address{Center for Mathematics and Interdisciplinary Sciences, Fudan University, Shanghai 200433, China \newline
${\quad}$ Shanghai Institute for Mathematics and Interdisciplinary Sciences (SIMIS), Shanghai 200433, China}
\curraddr{}
\email{yilicms@simis.cn, yilicms@gmail.com}

\author{Weijie Miao}
\address{Center for Mathematics and Interdisciplinary Sciences, Fudan University, Shanghai 200433, China  \newline
${\quad}$ Shanghai Institute for Mathematics and Interdisciplinary Sciences (SIMIS), Shanghai 200433, China }
\curraddr{}
\email{25114020012@fudan.edu.cn}
\thanks{}



\date{}

\dedicatory{}

\begin{abstract}
   In this paper, we prove the long-time existence of the heterotic $G_2$ flow with bounded Ricci curvature, torsion tensor and dilaton. Besides, we also get the forward and backward uniqueness results for the solution.
   
\end{abstract}

\maketitle

\section{Introduction}

Geometric flows play a crucial role in the study of geometric structures. The core idea is to start with a general geometric structure on a manifold and then utilize a flow to obtain a more specialized structure. Let $M$ be a $7$-dimensional manifold. Considering a closed $G_{2}$-structure on $M$, Bryant \cite{Bryant 2006} introduced the following Laplacian  flow:
\begin{equation}
  \left \{
       \begin{array}{rl}
          \partial_{ t}\varphi(t)&=\Delta_{\varphi(t)}\varphi(t),\\
           \varphi(0)&=\varphi,
       \end{array}
  \right.
  \label{The closed Laplacian flow}
\end{equation}
where $\Delta_{\varphi(t)}=dd^{\ast}_{\varphi(t)}+d^{\ast}_{\varphi(t)}d$ is the Hodge Laplacian of $g(t)$ and $\varphi$ is an initial closed $G_{2}$-structure. Here $g(t)$ is the associated Riemannian metric of $\varphi(t)$.
This flow  $\eqref{The closed Laplacian flow}$ is the gradient flow of the Hitchin functional \cite{Hitchin 2000}, whose critical points are torsion-free $G_{2}$-structure.
Bryant and Xu \cite{Bryant-Xu 2011} proved short-time existence and uniqueness within a fixed cohomology class.  Lotay and Wei \cite{Lotay-Wei Shi-estimate, L-W 2019, L-W 2019 2} established foundational analytic results, including derivative estimates, blow‑up characterisation, dynamical stability, and real analyticity; subsequent works \cite{Fine-Yao 2018, LL G2 flow, Picard-Suan} addressed blow‑up under bounded Ricci curvature, scalar curvature, or torsion.


For the coclosed $G_{2}$-structure, Karigiannis, McKay, and Tsui \cite{KMT2012} proposed  Laplacian coflow (with the opposite sign):
\begin{equation}
  \left \{
       \begin{array}{rl}
          \partial_{ t}\psi(t)&=\Delta_{\psi(t)}\psi(t),\\
           \psi(0)&=\psi,
       \end{array}
  \right.
  \label{The coflow}
\end{equation}
where $\Delta_{\psi(t)}$ is the Hodge Laplacian induced by $\psi(t)$ and $\psi$ is an initial coclosed $G_{2}$-structure. However, this equation is not weakly parabolic and its principal symbol is indefinite. To remedy this,  Grigorian \cite{Grigorian-2013} introduced the modified Laplacian coflow:
\begin{equation}
  \left \{
       \begin{array}{rl}
          \partial_{ t}\psi(t)&=\Delta_{\psi(t)}\psi(t)+2d\left[\left(A-{\rm tr} \big{(}\mathbf{T}(t)\big{)}\right)\varphi(t)\right],\\
           \psi(0)&=\psi,
       \end{array}
  \right.
  \label{modified Laplacian coflow}
\end{equation}
where ${\rm tr}\big{(}\mathbf{T}(t)\big{)}$ is the trace of the torsion $\mathbf{T}(t)$, $A$ is a nonnegative constant. {If restricted to coclosed $G_{2}$-structures($d\psi(t)=0$), then the modified Laplacian coflow $\eqref{modified Laplacian coflow}$ preserves the cohomology class of $\psi$. Under this condition,} Grigorian proved the short-time existence and uniqueness of the modified Laplacian coflow in \cite{Grigorian-2013}. Later, Chen \cite{Chen Shi-estimates} derived Shi‑type estimates and studied finite‑time singularities; Bedulli and Vezzoni \cite{BV2020} proved stability; Real-analyticity was proved by the authors in \cite{LL26B}. However, there is no characterization of the fixed points of modified coflow. Motivated by Ashmore, Minasian and Proto \cite{AMP24}, Karigiannis-Picard-Suan\cite{KPS26} and Garcia-Fernandez-Moreno-Payne-Streets\cite{GMPS26} independently posed a novel geometric flow of conformally coclosed $G_{2}$-structures with dilaton. The flow in \cite{KPS26} focuses on the dimensional reduction of the six-dimensional anomaly flow \cite{PPZ18A,PPZ18B} and a special dilaton function satisfies
$$e^{4f}=e^{4f_{0}}\frac{{\rm Vol}_{\varphi}}{{\rm Vol}_{0}},$$
where $f_{0}$ and ${\rm Vol}_{0}$ are the initial conditions. The flow in \cite{GMPS26} focuses on the dilaton function satisfying an evolution equation. In this paper, we study 
the heterotic $G_{2}$ flow in \cite{GMPS26}: 
\begin{equation}
  \left \{
       \begin{array}{rl}
           \partial_{t}(e^{-4\phi}\psi(t))&=-d H_{\varphi}-\displaystyle\frac{7}{3}d(\tau_{0}\varphi),\\
           \partial_{t}\phi&=e^{4\phi}\left(\Delta\phi-\gamma|d\phi|^{2}+\frac{1}{4}|\tau_{3}|^{2}-\sigma\tau_{0}^{2}\right),
       \end{array}
  \right.
  \label{heterotic flow}
\end{equation}
where  $\varphi(t)$ is the $G_{2}$-structures and $\psi(t)$ is its dual $4$-form, $\phi(t)\in C^{\infty}(M,\mathbb{R})$ is a smooth dilaton function, $H_{\varphi}$ is the torsion of an affine connection $\nabla^{+}$ with respect to $\varphi$ satisfies
$$H_{\varphi}=\frac{1}{6}\tau_{0}\varphi-\tau_{1}\lrcorner\psi-\tau_{3},$$
$\tau_{0},\tau_{1}$ and $\tau_{3}$ are intrinsic torsion forms of the $G_{2}$-structure $\varphi$, $\sigma,\gamma$ are constants. The conformally coclosed condition is preserved along this flow. The authors of \cite{GMPS26} prove short-time existence, Shi-type smoothing, and fixed-point classification for this flow, which also admits a monotonicity formula for the $G_{2}$-dilaton functional-together with the Shi-type estimates, this monotonicity yields convergence of nonsingular solutions.

\begin{remark}
    In \cite{KPS26}, the dilaton is determined by the volume ratio rather than governed by an independent evolution equation. Short-time existence is established only for $C<1$, while torsion-free fixed points can be concluded only for  $C>\frac{10}{7}$, which lies outside the existence range. 
\end{remark}

The authors of \cite{GMPS26} employ the Shi-type estimates to prove the following long-time existence theorem for the heterotic $G_{2}$ flow $\eqref{heterotic flow}$.
\begin{theorem}[Corollary 1.4,\cite{GMPS26}]\label{Thm1.2}
    Let $(\psi(t),\phi(t))_{t\in[0,T)}$ be a solution to the heterotic $G_{2}$ flow $\eqref{heterotic flow}$ on a compact manifold $M$ with initial condition $(\varphi_{0},\phi_{0})$ satisfying $d(e^{-4\phi_{0}}\psi_{0})=0$, where $T<+\infty$. If there exists a constant $\Lambda>0$ such that
$$|{\rm Rm}|+|\nabla\mathbf{T}|+|\mathbf{T}|^{2}+|\phi|\leq \Lambda,\quad\quad\quad\quad \text{on}\ M\times[0,T),$$
then the flow $\eqref{heterotic flow}$ can be extended past time $T$.
\end{theorem} 

Motivated by the method in \cite{CLN06} and \cite{Lotay-Wei Shi-estimate}, by pulling back the metric via a smooth family of diffeomorphisms, we can remove the bound on $\nabla \mathbf{T}$ and replace the full curvature bound by a Ricci curvature bound in Theorem \ref{Thm1.2}.
\begin{theorem}\label{thm1.3}
Let $(\psi(t),\phi(t))_{t\in[0,T)}$ be a solution to the heterotic $G_{2}$ flow $\eqref{heterotic flow}$ on a compact manifold $M$ with initial condition $(\varphi_{0},\phi_{0})$ satisfying $d(e^{-4\phi_{0}}\psi_{0})=0$, where $T<+\infty$. If there exists a constant $\Lambda>0$ such that
$$|{\rm Ric}|+|\mathbf{T}|^{2}+|\phi|\leq \Lambda,\quad\quad\quad\quad \text{on}\ M\times[0,T),$$
then the flow $\eqref{heterotic flow}$ can be extended past time $T$.
\end{theorem}



We also study the uniqueness of the solution to the heterotic $G_{2}$ flow $\eqref{heterotic flow}$. Motivated by \cite{Kot14}, we give a new proof of forward uniqueness, and prove the backward uniqueness by using the method in \cite{Kot10}.
\begin{theorem}
    Suppose $(\psi(t),\phi(t)),(\tilde{\psi}(t),\tilde{\phi}(t))$ are  two solutions to the heterotic $G_{2}$ flow $\eqref{heterotic flow}$ on a compact manifold $M$ for $t\in[0,\epsilon],\epsilon>0$. If $$(\psi(t),\phi(t))=(\tilde{\psi}(t),\tilde{\phi}(t))$$
    for some $t\in[0,\epsilon]$, then $(\psi(s),\phi(s))=(\tilde{\psi}(s),\tilde{\phi}(s))$ for all $s\in[0,\epsilon]$.
\end{theorem}

We give an outline of this paper. We review the basic theory in Section \ref{section2} about $G_{2}$-structure, $G_{2}$-decompositions of $2$-forms and $3$-forms, the torsion tensors of $G_{2}$-structures, and curvature tensor of closed $G_{2}$-structures. By pulling back the metric via a smooth family of diffeomorphisms, we prove the long-time existence with bounded Ricci curvature, torsion tensor and dilaton in Section \ref{section3}. In Section \ref{section4}, we prove the forward and backward uniqueness results of the heterotic $G_{2}$ flow $\eqref{heterotic flow}$.

\section{$G_{2}$-structures}\label{section2}

\subsection{$G_{2}$-structures on $7$-manifolds}
Let $\mathbb{O}$ be the octonions (exceptional division algebra). {From} the vector cross product ``$\times$" on ${\rm Im}\ \mathbb{O}$, we can define the 3-form by
$$\phi(a,b,c):=\frac{1}{2}\langle {[}a,b{]},c\rangle=\langle a\times b,c\rangle\quad\quad\quad\text{for}\ a,b,c\in {\rm Im}\ \mathbb{O}.$$
Let $\{e_{1},e_{2},\cdots,e_{7}\}$ denote the standard basis of $\mathbb{R}^{7}$ and $\{e^{1},e^{2},\cdots,e^{7}\}$ be its dual basis. Using the octonion multiplication table, one can show that
\begin{align}\label{phi}
    \phi=e^{123}+e^{145}+e^{167}+e^{246}-e^{257}-e^{347}-e^{356},
\end{align}
where $e^{ijk}:=e^{i}\wedge e^{j}\wedge e^{k}$. The subgroup {fixing $\phi$} of ${\rm GL}(7,\mathbb{R})$ is the exceptional Lie group $G_{2}$, which is a compact, connected, simple $14$-dimensional Lie subgroup of ${\rm SO}(7)$. In fact, {Lie group} $G_{2}$ acts irreducibly on $\mathbb{R}^{7}$ and preserves the metric and orientation for which $\{e_{1},e_{2},\cdots,e_{7}\}$ is an oriented orthonormal basis. Note that {Lie group} $G_{2}$ also preserves the $4$-form
$$\ast_{\phi}\phi=e^{4567}+e^{2367}+e^{2345}+e^{1357}-e^{1346}-e^{1256}-e^{1247},$$
where $\ast_{\phi}$ is the Hodge star operator determined by the metric and orientation.


For a smooth $7$-manifold $M$ and a point $x\in M$, we define 
\begin{align}
    \wedge_{+}^{3}(T_{x}^{\ast}M):=\left\{\varphi_{x}\in\wedge^{3}(T_{x}^{\ast}M)\ \Big{|}\ h^{\ast}\phi=\varphi_{x},\
    \text{for\ invertible}\ h\in\text{Hom}_{\mathbb{R}}(T_{x}M,\mathbb{R}^{7})\right\}\notag
\end{align}
and the bundle
\begin{align}
    \wedge_{+}^{3}(T^{\ast}M):=\bigcup_{x\in M}\wedge_{+}^{3}(T_{x}^{\ast}M)\notag.
\end{align}
We call a section $\varphi$ of $\wedge_{+}^{3}(T^{\ast}M)$ a {\it positive $3$-form} on $M$ or a {\it $G_{2}$-structure} on $M$, and denote the space of positive 3-form{s} by $\Omega^{3}_{+}(M)$. The existence of $G_{2}$-structure is equivalent to the
property that $M$ is orientable and spinnable, which is equivalent to the vanishing of the first two Stiefel-Whitney classes $w_{1}(TM)$ and $w_{2}(TM)$. 

For a $3$-form $\varphi$, we define a $\Omega^{7}(M)$-valued bilinear form $\text{B}_{\varphi}$ by
$$\text{B}_{\varphi}(u,v)=\frac{1}{6}(u\lrcorner\varphi)\wedge(v\lrcorner\varphi)\wedge\varphi,$$
where $u, v$ are tangent vectors on $M$ and $``\lrcorner"$ is the interior multiplication operator (\textit{Here we use the orientation in} \cite{Bryant 2006}, {which refers to $\eqref{phi}$}). Then we can see that any $\varphi\in\Omega^{3}_{+}(M)$ determines a Riemannian metric $g_{\varphi}$ and an orientation $d V_{\varphi}$, hence the Hodge star operator $\ast_{\varphi}$ and the associated $4$-form
$$\psi:=\ast_{\varphi}\varphi$$
can also be uniquely determined by $\varphi$. 

\subsection{$G_{2}$ decompositions and torsion}
The group $G_{2}$ acts irreducibly on $\mathbb{R}^{7}$ (and hence on $\wedge^{1}(\mathbb{R}^{7})^{\ast}$ and $\wedge^{6}(\mathbb{R}^{7})^{\ast}$), but it
acts reducibly on $\wedge^{k}(\mathbb{R}^{7})^{\ast}$ for $2\leq k\leq 5$. Hence a $G_{2}$ structure $\varphi$ induces splittings
of the bundles $\wedge^{k}(T^{\ast}M)(2\leq k\leq5)$ into direct summands, which we denote by
$\wedge^{k}_{l}(T^{\ast}M,\varphi)$ with $l$ being the rank of the bundle. We let the space of sections
of $\wedge^{k}_{l}(T^{\ast}M,\varphi)$ be $\Omega^{k}_{l}(M)$. Define the natural projections
$$\pi^{k}_{l}:\Omega^{k}(M)\longrightarrow \Omega^{k}_{l}(M),\ \ \alpha\longmapsto \pi^{k}_{l}(\alpha).$$
Then we have
\begin{align}
    \Omega^{2}(M)&=\Omega^{2}_{7}(M)\oplus\Omega^{2}_{14}(M),\notag\\
    \Omega^{3}(M)&=\Omega^{3}_{1}(M)\oplus\Omega^{3}_{7}(M)\oplus\Omega^{3}_{27}(M)\notag{,}
\end{align}
where each component is determined by
\begin{align}
    \Omega^{2}_{7}(M)&=\{X\lrcorner\varphi:X\in C^{\infty}(TM)\}=\{\beta\in\Omega^{2}(M):\ast_{\varphi}(\varphi\wedge\beta)=2\beta\},\notag\\
    \Omega^{2}_{14}(M)&=\{\beta\in\Omega^{2}(M):\psi\wedge\beta=0\}=\{\beta\in\Omega^{2}(M):\ast_{\varphi}(\varphi\wedge\beta)=-\beta\},\notag
\end{align}
and
\begin{align}
    \Omega^{3}_{1}(M)&=\{f\varphi:f\in C^{\infty}(M)\},\notag\\
    \Omega^{3}_{7}(M)&=\{\ast_{\varphi}(\varphi\wedge\alpha):\alpha\in\Omega^{1}(M)\}=\{X\lrcorner\psi:X\in C^{\infty}(TM)\},\notag\\
    \Omega^{3}_{27}(M)&=\{\eta\in\Omega^{3}(M):\eta\wedge\varphi=\eta\wedge\psi=0\}.\notag
\end{align}

\begin{remark}\label{remark2.1}
    $\Omega^{4}$ and $\Omega^{5}$ have the corresponding decompositions by Hodge duality. {For} more details about $G_{2}$-decompositions, see \cite{Bryant 2006, Spiros introduce to G2}.
\end{remark}

By the definition  of $G_{2}$ decompositions, we can find unique differential forms
$\tau_{0}\in\Omega^{0}(M),\tau_{1},\widetilde{\tau}_{1}\in\Omega^{1}(M),\tau_{2}\in\Omega^{2}_{14}(M)$ and $\tau_{3}\in\Omega^{3}_{27}(M)$ such that (see \cite{Bryant 2006})
\begin{align}\label{torsion1}
    d\varphi&=\tau_{0}\psi+3\!\ \tau_{1}\wedge\varphi+\ast_{\varphi}\tau_{3},\\
    d\psi&=4\!\ \widetilde{\tau}_{1}\wedge\psi+\tau_{2}\wedge\varphi.
\end{align}
In fact, Bryant \cite{Bryant 2006} proved that $\tau_{1}=\widetilde{\tau}_{1}$. We call $\tau_{0}$ the \textit{scalar torsion}, $\tau_{1}$ the \textit{vector torsion}, $\tau_{2}$ the \textit{Lie algebra torsion}, and $\tau_{3}$ the \textit{symmetric traceless torsion}. We also call $\tau_{\varphi}:=\{\tau_{0},\tau_{1},\tau_{2},\tau_{3}\}$ the intrinsic torsion forms of the $G_{2}$-structure $\varphi$.

{From \cite{Spiros introduce to G2}, the covariant derivative $\nabla\varphi$ is a smooth section of $T^{\ast}M\otimes \Lambda^{3}_{7}(T^{\ast}M)$. Thus we can define the full torsion tensor as 
\begin{definition}
    Let $X$ be a vector field on $M$. $\nabla_{X}\varphi$ can be written as
    $$\nabla_{X}\varphi=\mathbf{T}(X)\lrcorner \psi$$
    for some vector field $\mathbf{T}(X)$ on $M$. We call $\mathbf{T}=\mathbf{T}_{ij}dx^{i}\otimes dx^{j}$ the full torsion tensor of $\varphi$, which satisfies
    \begin{align}\label{2.7}
    \nabla_{i}\varphi_{jkl}&=\mathbf{T}_{i}^{\ m}\psi_{mjkl},\\
    \mathbf{T}_{i}^{\ j}&=\frac{1}{24}\nabla_{i}\varphi_{lmn}\psi^{jlmn},
\end{align}
and 
\begin{align}
    \nabla_{m}\psi_{ijkl}=-(\mathbf{T}_{mi}\varphi_{jkl}-\mathbf{T}_{mj}\varphi_{ikl}-\mathbf{T}_{mk}\varphi_{jil}-\mathbf{T}_{ml}\varphi_{jki}).
\end{align}
\end{definition}
The full torsion tensor $\mathbf{T}_{ij}$ is related to the
intrinsic torsion forms by the following:
\begin{align}
\label{Def of T}\mathbf{T}_{ij}=\frac{\tau_{0}}{4}g_{ij}-(\tau_{3})_{ij}-(\tau_{1}^{\#}\lrcorner\varphi)_{ij}-\frac{1}{2}(\tau_{2})_{ij}
\end{align}
or as $2$-tensors,
\begin{align}\label{Tdecomposition}
\mathbf{T}=\frac{\tau_{0}}{4}g_{\varphi}-\tau_{3}-\tau_{1}^{\#}\lrcorner\varphi-\frac{1}{2}\tau_{2},
\end{align}
where $(\tau_{1}^{\#}\lrcorner\varphi)_{ij}=(\tau_{1}^{\#})^{l}\varphi_{lij}$ and $\#$ is the isomorphism from $1$-form to vector fields.}

\subsection{Heterotic $G_{2}$ flow}

A $G_2$-structure is called integrable if $\tau_2=0$. Such a structure admits a unique metric connection with totally skew-symmetric torsion preserving $\varphi$. The characteristic connection and its torsion are \cite[Proposition 2.3]{GMPS26}
\begin{equation}\label{characteristic}
 \nabla^+=\nabla+\frac12g^{-1}H_\varphi,
 \qquad H_\varphi=\frac16\tau_0\phi-\tau_1^\sharp\lrcorner\psi-\tau_3.
\end{equation}
More explicitly,
\[
 g(\nabla^+_XY,Z)=g(\nabla_XY,Z)+\frac12H_\varphi(X,Y,Z).
\]

Let $\phi\in C^\infty(M,\mathbb{R})$. We say that $\varphi$ is conformally coclosed if
\begin{equation}\label{weightedclosed}
 d(e^{-4\phi}\psi)=0.
\end{equation}
From \eqref{torsion1},
\[
 d(e^{-4\phi}\psi)=e^{-4\phi}
 \big(4(\tau_1-d\phi)\wedge\psi+\tau_2\wedge\varphi\big).
\]
Consequently,
\begin{equation}\label{conformaltorsion}
 d(e^{-4\phi}\psi)=0
 \quad\Longleftrightarrow\quad \tau_1=d\phi,\quad\tau_2=0.
\end{equation}

We record some estimates used below. For a conformally coclosed pair, the orthogonal decomposition gives
\begin{equation}\label{torsionnorm}
 |\mathbf{T}|^2=\frac7{16}\tau_0^2+6|d\phi|^2+\frac12|\tau_3|^2.
\end{equation}
Indeed, the squared tensor norm of $(d\phi)^\sharp\lrcorner\varphi$ is $6|d\phi|^2$, and $|\tau_{27}|^2=\frac12|\tau_3|^2$; see \cite[Section 2.2]{GMPS26}. The orthogonal components of $H_\varphi$ similarly give
\[
 |H_\varphi|^2=\frac7{36}\tau_0^2+4|d\phi|^2+|\tau_3|^2.
\]
Therefore $|H_\varphi|+|d\phi|+|\tau_0|+|\tau_3|\leq C|\mathbf{T}|$.

Contracting \eqref{Tdecomposition} with $\varphi_k{}^{ij}$ and using $\varphi_{pij}\varphi_k{}^{ij}=6g_{pk}$, we obtain
\begin{equation}\label{dilatoncontraction}
 \nabla_k\phi=-\frac16\mathbf{T}_{ij}\varphi_k{}^{ij}.
\end{equation}
Differentiating this identity yields
\begin{equation}\label{dilatonhessian}
 \nabla_l\nabla_k\phi
 =-\frac16(\nabla_l\mathbf{T}_{ij})\varphi_k{}^{ij}
  -\frac16\mathbf{T}_{ij}\mathbf{T}_l{}^p\psi_{pk}{}^{ij}.
\end{equation}
In particular,
\begin{equation}\label{dilatonbounds}
 |d\phi|\leq C|\mathbf{T}|,\qquad
 |\nabla^2\phi|\leq C\big(|\nabla \mathbf{T}|+|\mathbf{T}|^2\big).
\end{equation}

For completeness, the Ricci tensor can also be expressed in terms of torsion as \cite[Lemma 2.2]{GMPS26}
\begin{equation}\label{RicciT}
 \begin{split}
  R_{ij}={}(\nabla_i\mathbf{T}_{mn}-\nabla_m\mathbf{T}_{in})\varphi_j{}^{mn}
 +({\rm tr} \mathbf{T})\mathbf{T}_{ij}-\mathbf{T}_{im}\mathbf{T}^m{}_j+\mathbf{T}_{im}\mathbf{T}_{np}\psi^{mnp}{}_j.
 \end{split}
\end{equation}
Thus $|{\rm Ric}|\leq C(|\nabla \mathbf{T}|+|\mathbf{T}|^2)$. 

We now return to the heterotic $G_2$ flow. The right-hand side of its first equation is exact, so
\[
 \partial_t d(e^{-4\phi}\psi)=0,
 \qquad [e^{-4\phi(t)}\psi(t)]=[e^{-4\phi_0}\psi_0].
\]
Hence \eqref{conformaltorsion} holds throughout a solution with conformally coclosed initial data. 
It is the generic flow in \cite{GMPS26} with their parameter $C=-4/3$. For this choice, short-time existence holds for every fixed $\gamma,\sigma\in\mathbb{R}$ by \cite[Theorem 4.8]{GMPS26}.

The induced metric evolves by \cite[Proposition 3.2]{GMPS26}
\begin{equation}\label{metricevolution}
 \partial_tg=e^{4\phi}\big(-2{\rm Ric}-8\nabla^2\varphi+E\big),
\end{equation}
where
\begin{equation}\label{Edefinition}
 E=\frac12H_\varphi^2-\frac{14}{3}\tau_0\mathbf{T}_{\rm sym}
 +\left[\left(\frac7{12}-2\sigma\right)\tau_0^2
        +2(3-\gamma)|d\phi|^2\right]g.
\end{equation}
Here $(H_\varphi^2)_{ij}=(H_\varphi)_{ipq}(H_\varphi)_j{}^{pq}$ and $\mathbf{T}_{\rm sym}=\frac12(\mathbf{T}+\mathbf{T}^{\mathsf t})$. Formula \eqref{torsionnorm} gives
\begin{equation}\label{Ebound}
 |E|\leq C(\gamma,\sigma)|\mathbf{T}|^2.
\end{equation}
We use the usual Lie derivative convention
\[
 \mathcal L_{\nabla\phi}g=2\nabla^2\phi;
\]
thus the Hessian term in \eqref{metricevolution} is also $-4e^{4\phi}\mathcal L_{\nabla\phi}g$.

\section{Long time existence of the heterotic $G_{2}$ flow}\label{section3}
In this section, we give the long-time existence of the heterotic $G_{2}$ flow assuming that Ricci curvature, torsion tensor and dilaton are uniformly bounded.

\begin{lemma}[Shi-type estimate,\cite{GMPS26}]\label{lemma3.3}
    Let $(\psi(t),\phi(t))$ be a solution to the heterotic $G_{2}$ flow $\eqref{heterotic flow}$ on a compact manifold $M$ with initial condition $(\varphi_{0},\phi_{0})$ satisfying $d(e^{-4\phi_{0}}\psi_{0})=0$. Let $B_{r}(p)$ be a ball of radius $r$ around $p\in M$ with respect to the metric $g_{0}$. Suppose that there exist $t_{0},\Lambda>0$ such that
    $$|{\rm Rm}|+|\nabla\mathbf{T}|+|\mathbf{T}|^{2}+|\phi|\leq \Lambda,\quad\quad\quad\quad \text{on}\ M\times[0,T)$$
    on $B_{r}(p)\times [0,t_{0}]$. Then, for each $k\geq 0$, there is a constant $C(k,t_{0},\Lambda,\gamma,\sigma,r)$ such that
    $$|\nabla^{k}{\rm Rm}|+|\nabla^{k+1}\mathbf{T}|+|\nabla^{k+2}\phi|<C(k,t_{0},\Lambda,\gamma,\sigma,r)$$
    on $B_{\frac{r}{2}}(p)\times [\frac{t_{0}}{2},t_{0}]$.
\end{lemma}

\begin{lemma}\label{lemma3.1}
Let $(\psi(t),\phi(t))_{t\in[0,T)}$ be a solution to the heterotic $G_{2}$ flow $\eqref{heterotic flow}$ on a compact manifold $M$ with initial condition $(\varphi_{0},\phi_{0})$ satisfying $d(e^{-4\phi_{0}}\psi_{0})=0$, where $T<+\infty$. If there exists a constant $\Lambda>0$ such that
$$|{\rm Ric}|+|\mathbf{T}|^{2}+|\phi|\leq \Lambda,\quad\quad\quad\quad \text{on}\ M\times[0,T),$$
then there exists a smooth family of diffeomorphisms
$\chi_{t}:M\rightarrow M$ defined on $[0,T)$ with $\chi_{0}={\rm id}$ such that
$$\hat{g}(t):=\chi^{\ast}_{t}g(t)$$
satisfies $|\partial_{t}\hat{g}(t)|_{\hat{g}(t)}\leq L$ for a constant $L$ independent of $t< T$.
\end{lemma}
\begin{proof}
    Let $W_{t}$ be a vector field satisfies
    $$W_{t}=\nabla (e^{4\phi})=4e^{4\phi}\nabla\phi,$$
    we observe that $W_{t}$ is smooth on $M\times [0,T)$. By solving an ODE
   \begin{equation}
   \left \{
       \begin{array}{rl}
           \partial_{t}\chi_{t}&=W_{t}\circ\chi_{t},\\
           \chi_{0}&={\rm id},
       \end{array}
  \right.
\end{equation}
we get a smooth family of diffeomorphisms $\{\chi_{t}\}_{t}$. Define
$$\hat{g}(t):=\chi^{\ast}_{t}g(t),$$
it follows that
\begin{align}
    \partial_{t}\hat{g}&=\partial_{t}(\chi^{\ast}_{t}g)=\chi^{\ast}_{t}(\partial_{t}g+\mathcal{L}_{W_{t}}g)\notag\\
    &=\chi_{t}^{\ast}[e^{4\phi}(-2{\rm Ric}-8\nabla^{2}\phi+E)+8e^{4\phi}\nabla^{2}\phi+32e^{4\phi}d\phi\otimes d\phi]\notag\\
    &=\chi_{t}^{\ast}[e^{4\phi}(-2{\rm Ric}+E+32d\phi\otimes d\phi)]\notag,
\end{align}
which means
$$|\partial_{t}\hat{g}|\leq L,$$
the inequality uses the assumption and $L$ is a constant independent of $t< T$.
\end{proof}

\begin{lemma}\label{lemma3.2}
Let $\{g(t)\}_{t\in[0,T)}$ be a family of smooth metrics on a closed $n$ dimensional manifold with $T<+\infty$. Suppose that $|\partial_{t}g(t)|_{g(t)}\leq L$ with a constant $L$ independent of $t< T$. If there exists a sequence $\{t_{i},y_{i},r_{i}\}_{i}$ with 
$$t_{i}\rightarrow T,\quad y_{i}\in M,\quad r_{i}\rightarrow 0\quad \text{as}\quad i\rightarrow \infty,$$
then we have
$$\lim_{i\rightarrow\infty}\frac{{\rm Vol}_{g(t_{i})}(B_{g(t_{i})}(y_{i},r_{i}))}{\omega_{n}r_{i}^{n}}=1,$$
where $B_{g(t_{i})}(y_{i},r_{i})$ is a geodesic ball and $\omega_{n}$ is the volume of a geodesic ball of radius $1$ in Euclidean space.
\end{lemma}
\begin{proof}
    From the assumption and fix the time $\tau\in[0,T)$, for $i$ sufficiently large such that $t_{i}\geq \tau$, we have
    $$\left|\int_{\tau}^{t_{i}}\partial_{t}g(t)\ dt\right|\leq \int_{\tau}^{t_{i}}|\partial_{t}g(t)|\ dt\leq L(t_{i}-\tau).$$
    Let $\delta:=L(T-\tau)$. Thus, from Lemma 6.49 in \cite{Chow-Knopf 2004}, we arrive at
    $$e^{-\delta}g(\tau)\leq e^{-L(t_{i}-\tau)}g(\tau)\leq g(t_{i})\leq e^{L(t_{i}-\tau)}g(\tau)\leq e^{\delta}g(\tau),$$
    which implies
    $$B_{g(\tau)}(y_{i},e^{-\frac{\delta}{2}}r_{i})\subset B_{g(t_{i})}(y_{i},r_{i})\subset B_{g(\tau)}(y_{i},e^{\frac{\delta}{2}}r_{i}).$$
    Hence, the volume of the geodesic ball satisfies
    \begin{align}
        e^{-\frac{n\delta}{2}}{\rm Vol}_{g(\tau)}(B_{g(\tau)}(y_{i},e^{-\frac{\delta}{2}}r_{i}))&\leq {\rm Vol}_{g(t_{i})}(B_{g(t_{i})}(y_{i},r_{i}))\\
        &\leq e^{\frac{n\delta}{2}}{\rm Vol}_{g(\tau)}(B_{g(\tau)}(y_{i},e^{\frac{\delta}{2}}r_{i}))\notag.
    \end{align}
    The smooth fixed metric $g(\tau)$ on the compact manifold has uniformly Euclidean small-ball volume ratios, uniformly in the center. It follows that
    \begin{align}
        e^{-n\delta}&\leq \liminf_{i\rightarrow\infty}\frac{{\rm Vol}_{g(t_{i})}(B_{g(t_{i})}(y_{i},r_{i}))}{\omega_{n} r_{i}^{n}}\notag\\
        &\leq \limsup_{i\rightarrow\infty}\frac{{\rm Vol}_{g(t_{i})}(B_{g(t_{i})}(y_{i},r_{i}))}{\omega_{n} r_{i}^{n}}\notag\\
        &\leq e^{n\delta}\notag.
    \end{align}
   Let $\tau\rightarrow T$, then we prove this lemma.
\end{proof}

\begin{lemma}\label{lemma3.4}
Let $(N, g)$ be complete and let $A$ be a smooth tensor field satisfying
\[
|A|_{g} \le c_{0}, \qquad |\nabla_{g}^2 A|_{g} \le c_{2}.
\]
Then
$$
|\nabla_{g} A|^2_{g} \le 2nc_{0} c_{2}.
$$
where $c_{0},c_{2}$ are nonnegative constants and $n$ is the dimension of manifold. 
\end{lemma}

\begin{proof}
Fix $p\in N$ and a unit vector $v\in T_pN$. By completeness, there is a unit-speed geodesic $\gamma:\mathbb{R}\to N$ with $\gamma(0)=p$ and $\dot\gamma(0)=v$. Choose any unit tensor $B_0$ of the same type as $A(p)$, and let $B(s)$ be its parallel transport along $\gamma$. Set
\[
 f(s)=\langle A(\gamma(s)),B(s)\rangle.
\]
Since $\gamma$ is a geodesic and $B$ is parallel, we have
\[
 f'(s)=\langle \nabla_{\dot\gamma}A,B(s)\rangle,\qquad
 f''(s)=\langle\nabla_{\dot\gamma}\nabla_{\dot\gamma}A,B(s)\rangle.
\]
It follows that $|f(s)|\leq c_{0}$ and $|f''(s)|\leq c_{2}$ for every $s\in\mathbb{R}$.

For any $r>0$, Taylor's formula with integral remainder gives
\begin{align*}
 f(r)&=f(0)+rf'(0)+\int_0^r(r-s)f''(s)\,ds,\\
 f(-r)&=f(0)-rf'(0)+\int_0^r(r-s)f''(-s)\,ds.
\end{align*}
Subtracting the two equalities and estimating the remainders, we obtain
\[
 2r|f'(0)|\leq|f(r)-f(-r)|+\int_0^r(r-s)\big(|f''(s)|+|f''(-s)|\big)\,ds
 \leq2 c_{0}+c_{2}r^2.
\]
Thus
\[
 |f'(0)|\leq\frac{c_{0}}{r}+\frac{c_{2}r}{2}.
\]
If $c_{0}=0$, then $A=0$ and the result is immediate. If $c_{2}=0$, letting $r\to\infty$ gives $f'(0)=0$. If both constants are positive, choose $r=\sqrt{2c_{0}/c_{2}}$. We obtain
\[
 \langle\nabla_vA(p),B_0\rangle^2=|f'(0)|^2
 \leq2c_{0} c_{2}.
\]
Taking the supremum over unit tensors $B_0$ gives
\[
 |\nabla_vA(p)|^2\leq2c_{0} c_{2}.
\]
Now let $\{e_1,\ldots,e_n\}$ be an orthonormal basis of $T_pN$. The full tensor norm is
\[
 |\nabla A(p)|^2=\sum_{j=1}^n|\nabla_{e_j}A(p)|^2
 \leq2nc_{0} c_{2}.
\]
Taking the supremum over $p\in N$ proves the estimate.
\end{proof}

\begin{theorem}\label{thm2.1}
Let $(\psi(t),\phi(t))_{t\in[0,T)}$ be a solution to the heterotic $G_{2}$ flow $\eqref{heterotic flow}$ on a compact manifold $M$ with initial condition $(\varphi_{0},\phi_{0})$ satisfying $d(e^{-4\phi_{0}}\psi_{0})=0$, where $T<+\infty$. If there exists a constant $\Lambda>0$ such that
$$|{\rm Ric}|+|\mathbf{T}|^{2}+|\phi|\leq \Lambda,\quad\quad\quad\quad \text{on}\ M\times[0,T),$$
then the flow $\eqref{heterotic flow}$ can be extended past time $T$.
\end{theorem}
\begin{proof}
    We prove this theorem by contradiction. If 
    $$\limsup_{t\rightarrow T} \max_{M}(|{\rm Ric}|+|\mathbf{T}|^{2}+|\phi|)< +\infty,$$
the flow $\eqref{heterotic flow}$ can not be extended past time $T$. From Theorem \ref{Thm1.2}, the Riemann curvature and $|\nabla\mathbf{T}|$ must blow up, which means
$$\limsup_{t\rightarrow T} \max_{M}(|{\rm Rm}|+|\nabla\mathbf{T}|)=+\infty.$$
There exists a sequence of points and time $(x_{i},t_{i})$ with $t_{i}\rightarrow T$ as $i\rightarrow\infty$, such that
$$Q(x_{i},t_{i}):=\sup_{(x,t)\in M\times[0,t_{i}]}(|{\rm Rm}|_{g(t)}+|\nabla\mathbf{T}|_{g(t)})\rightarrow\infty.$$
Setting 
$$\varphi_{i}(t)=Q(x_{i},t_{i})^{\frac{3}{2}}\varphi(t_{i}+Q(x_{i},t_{i})^{-1}t),$$
then
$$g_{i}(t)=Q(x_{i},t_{i})g(t_{i}+Q(x_{i},t_{i})^{-1}t).$$
Thus we get a sequence of flows $\{(M,\varphi_{i}(t),\phi_{i},x_{i})\}_{i}$ defined on $[-t_{i}Q(x_{i},t_{i}),0]$ with
$$\sup_{M\times [-t_{i}Q(x_{i},t_{i}),0]}|Q(x,t)|\leq 1,\quad |Q(x_{i},0)|= 1.$$
Let $\chi_{t}$ and $\hat{g}(t)$ be as in Lemma \ref{lemma3.1}, and
put $y_{i} = \chi^{-1}_{t_{i}}(x_{i})$. Since $\chi_{t_{i}}$ is an isometry from $(M,\hat{g}(t_{i}))$ to $(M,g(t_{i}))$, according to Lemma \ref{lemma3.2}, we have
\begin{align}\label{3.5}
    \frac{{\rm Vol}_{g_{i}(0)}(B_{g_{i}(0)}(x_{i},r))}{\omega_{7}r^{7}}=\frac{{\rm Vol}_{\hat{g}(t_{i})}(B_{\hat{g}(t_{i})}(y_{i},r_{i}/\sqrt{Q_{i}}))}{\omega_{7}(r/\sqrt{Q_{i}})^{7}}\longrightarrow 1\quad \text{as}\quad i\rightarrow\infty
\end{align}
for each fixed $r>0$. In particular, the volumes of the unit balls centered at $x_{i}$ are uniformly bounded below. With the same discussion of \cite{Lotay-Wei Shi-estimate} and \cite{CLN06}, we get
$${\rm inj}(M,g_{i}(0),x_{i})\geq c>0.$$
Recall that
$$\sup_{M\times [-t_{i}Q(x_{i},t_{i}),0]}(|{\rm Rm}|(x,t)+|\nabla\mathbf{T}|(x,t))\leq 1,$$
from Corollary 1.4 in \cite{GMPS26}, $(M,\varphi_{i}(t),\phi_{i}(t),x_{i})$ converges to an ancient solution $$(M_{\infty},\varphi_{\infty}(t),\phi_{\infty}(t),x_{\infty})$$
with $t\in(-\infty,0]$ and $|Q_{\infty}(x_{\infty},0)|=1$. We know $|{\rm Ric}|$ is bounded from the assumptions, then
\begin{align}
    |\mathbf{T}_{i}(x,t)|^{2}_{g_{i} (t)}&=Q(x_{i},t_{i})^{-1}|\mathbf{T} (x,t_{i}+Q(x_{i},t_{i})^{-1}t)|^{2}_{g(t_{i}+Q(x_{i},t_{i})^{-1}t)}\longrightarrow 0\notag\\
    |{\rm Ric}_{i}(x,t)|_{g_{i}(t)}&=Q(x_{i},t_{i})^{-1}|{\rm Ric}(x,t_{i}+Q(x_{i},t_{i})^{-1}t)|_{g(t_{i}+Q(x_{i},t_{i})^{-1}t)}\longrightarrow 0\notag
\end{align}
as $i\rightarrow\infty$,
which means 
$${\rm Ric}_{g_{\infty}(t)}\equiv 0\quad  \text{for all}\quad t\in(\infty,0].$$ 
Apply Lemma \ref{lemma3.3} after shifting $[-1,0]$ to $[0,1]$, we get
$$\sup_{M}|\nabla_{g_{i}}^{2}\mathbf{T}_{i}|(x_{i},0)\leq C$$
for some constant $C$. by Lemma \ref{lemma3.4}, 
$$\sup_{x\in M}|\nabla_{g_{i}}\mathbf{T}_{i}|^{2}(x,0)\leq 14\sup_{x\in M}|\nabla_{g_{i}}^{2}\mathbf{T}_{i}|(x,0)\cdot \sup_{x\in M}|\mathbf{T}_{i}|(x,0)\longrightarrow 0\quad \text{as}\quad i\rightarrow\infty.$$
Thus, we obtain
$$|{\rm Rm}_{g_{\infty}}(x_{\infty},0)|= 1.$$

On the other hand, $\eqref{3.5}$ passes to the pointed smooth limit and gives
$${\rm Vol}_{g_{\infty}(0)}(B_{g_{\infty}(0)}(x_{\infty},r))=\omega_{7}r^{7}\quad \text{for every}\quad  r>0.$$
Since $g_{\infty}(x,0)$ is complete and Ricci-flat, the Bishop-Gromov comparison theorem implies 
$(M_{\infty},g_{\infty}(x,0))$ is flat, which contradicts $|{\rm Rm}_{g_{\infty}}(x_{\infty},0)|=1.$
\end{proof}

\section{Uniqueness of the heterotic $G_{2}$ flow}\label{section4}
In this section, we study forward and backward uniqueness for the heterotic $G_2$-flow.
\subsection{Forward uniqueness}
Suppose we have two smooth solutions
to the heterotic $G_2$ flow $(\psi(t),\phi(t))$ $(\tilde\psi(t),\tilde\phi(t))$ on a compact seven dimensional
manifold $M^7$ for $t\in[0,\epsilon]$
with $\epsilon>0$. 
Since both solutions are smooth on the compact cylinder $M\times[0,\epsilon]$, their curvature and torsion tensors, together with all covariant derivatives, are uniformly bounded. The constants in this section may depend on the two fixed solutions and on $\epsilon$. For each integer $k\geq 0$ there exists a uniform  constant $C_k$ such that
\begin{align}\label{uniformbdd}
    |\nabla^{k}\mathrm{Rm}|_{g(t)}
    +|\tilde{\nabla}^{k}\widetilde{\mathrm{Rm}}|_{\tilde g(t)}+|\phi|_{g(t)}
    +|\nabla^{k+1}\mathbf{T}|_{g(t)}
    +|\tilde{\nabla}^{k+1}\tilde{\mathbf{T}}|_{\tilde g(t)}+|\tilde\phi|_{\tilde g(t)}
    \leq C_k 
\end{align}
on $M\times[0,\epsilon]$. Integrating the evolution equation for $g(t)$ in
time and using \eqref{uniformbdd}, we conclude that $g(t)$ and $\tilde g(t)$
are uniformly equivalent on $M\times[0,\epsilon]$. Consequently, the norms
$|\cdot|_{g(t)}$ and $|\cdot|_{\tilde g(t)}$ are uniformly comparable, i.e.
there is a constant $C$, independent of $(x,t)\in M\times[0,\epsilon]$, such
that
$$
    C^{-1}\,|\cdot|_{\tilde g(t)} \leq |\cdot|_{g(t)} \leq C\,|\cdot|_{\tilde g(t)} .
$$
For later use we record the following consequence of \eqref{uniformbdd}.

\begin{lemma}
The tensors $\tilde g^{-1}$, $\tilde{\nabla}^{k}\widetilde{\mathrm{Rm}}$ and
$\tilde{\nabla}^{k}\tilde{\mathbf{T}}$, $k\geq 0$ are uniformly bounded with
respect to the metric $g(t)$ on $M\times[0,\epsilon]$.
\end{lemma}

Throughout, $A\ast B$ denotes any contraction of the tensors $A$ and $B$ with
respect to $g(t)$. We now prove the forward uniqueness of the heterotic
$G_2$ flow.
Following the strategy of \cite{Kot14}, we consider the energy functional
\begin{align}
    \mathcal{W}(t)=&\int_M e^{-4\phi}|\Phi|^{2}\mathrm{dVol}_{g(t)},\label{Wfunc}
\end{align}
where $|\Phi|^{2}$ is defined by
$$|\Phi|^{2}:=|P|^{2}+|F|^{2}+|Q|^{2}+|h|^{2}+|A|^{2}+|U|^{2}+|V|^{2}+|S|^{2},$$
and
\begin{align*}
    P &= \psi - \tilde{\psi},\ \quad F =e^{4\phi}-e^{4\tilde{\phi}},\quad  Q= \varphi - \tilde{\varphi}, 
\ \ \ \ \ \ \quad h= g - \tilde{g}, \\
A &= \nabla - \tilde{\nabla},\quad   U = \mathbf{T} - \tilde{\mathbf{T}}, \ \ \ \ \quad V = \nabla \mathbf{T} - \tilde{\nabla}\tilde{\mathbf{T}}, 
\quad S = \rm Rm - \widetilde{\rm Rm}.
\end{align*}
In local coordinates, 
\begin{align*}
    A^k_{ij}&=\Gamma^k_{ij}-\tilde \Gamma^k_{ij},~U_{ij}=\mathbf{T}_{ij}-\tilde{\mathbf{T}}_{ij},\\
    ~V_{ijk}&=\nabla_i\mathbf{T}_{jk}-\tilde{\nabla}_i\tilde{\mathbf{T}}_{jk}\text{ and }S_{ijk}^l=R_{ijk}^l-\tilde{R}_{ijk}^l.
\end{align*} 

Before we prove the forward uniqueness, we recall the evolution equations of Riemann curvature and torsion tensor in \cite{GMPS26}.
\begin{lemma}
Under the heterotic $G_{2}$ flow $\eqref{heterotic flow}$, the evolution equation of Riemann curvature and torsion tensor are 
\begin{align}
     e^{-4\phi}\frac{\partial}{\partial t}{\rm Rm}&=\Delta{\rm Rm}+\nabla^{2}\mathbf{T}\ast\mathbf{T}+{\rm Rm}\ast {\rm Rm}+{\rm Rm}\ast{\rm Ric}+\nabla{\rm Rm}\ast\mathbf{T}+{\rm Rm}\ast\nabla\mathbf{T}\notag\\
        &\quad+{\rm Rm}\ast\mathbf{T}\ast\mathbf{T}+\nabla\mathbf{T}\ast \nabla\mathbf{T}+\nabla\mathbf{T}\ast\mathbf{T}\ast\mathbf{T}+\mathbf{T}\ast\mathbf{T}\ast\mathbf{T}\ast\mathbf{T}\notag,\notag\\
        e^{-4\phi}\frac{\partial}{\partial t}\mathbf{T}&=\Delta \mathbf{T}+\mathbf{T}\ast({\rm Rm}+\nabla\mathbf{T}+\mathbf{T}\ast\mathbf{T})\ast(1+\varphi+\psi).\notag
\end{align}
\end{lemma}

To obtain a differential inequality for $\mathcal{W}(t)$, we next estimate the evolution of $|\Phi|$ appearing in the integrand defining $\mathcal{W}(t)$.

\begin{lemma}\label{lemma5.2}
    The following pointwise estimate holds:
\begin{align}
\frac{\partial}{\partial t}|\Phi|^{2}&\leq 2e^{4\phi}\langle V,(\Delta V(t)+{\rm div}\mathfrak{V}(t))\rangle+2e^{4\phi}\langle S,(\Delta S(t)+{\rm div}\mathfrak{S}(t))\rangle\\
&\quad+e^{4\phi}|\nabla V|^{2}+e^{4\phi}|\nabla S|^{2}+C|\Phi|^{2},\notag
\end{align}
where $\mathfrak{S}$ is defined by $\mathfrak{S}^{id}_{abc}:=(g^{ij}\nabla_j-\tilde{g}^{ij}\tilde{\nabla}_{j})\tilde R^d_{abc}$ satisfies
\begin{align}
    |\mathfrak{S}(t)|_{g(t)}\leq C\left(|h(t)|_{g(t)}+|A(t)|_{g(t)} \right),
\end{align}
and $\mathfrak{V}$ is defined by $\mathfrak{V}^{i}_{abc}:=(g^{ij}\nabla_j-\tilde{g}^{ij}\tilde{\nabla}_{j})\tilde \nabla_a \tilde{\mathbf{T}}_{bc}$ satisfies
\begin{align}
    |\mathfrak{V}(t)|_{g(t)}\leq C\left(|h(t)|_{g(t)}+|A(t)|_{g(t)} \right),
\end{align}
$C$ is a uniform constant.
\end{lemma}
\begin{proof}
We recall the evolution equations for $g(t)$, $\psi(t)$, and $\varphi(t)$ in \cite{GMPS26}.
The metric $g(t)$ evolves by
\begin{equation}
\partial_t g = e^{4\phi}\left(-2\,\mathrm{Ric}_g - 4\,\mathcal{L}_{\mathrm{d}\phi} g + E\right),
\end{equation}
where
\begin{equation}
E = \frac{1}{2}H_\varphi^2 - \frac{14}{3}\tau_0 \mathbf{T}_{\mathrm{sym}}
+ \left(\left(\frac{7}{12}-2\sigma\right)\tau_0^2 + 2(3-\gamma)|\mathrm{d}\phi|^2\right)g.
\end{equation}
The forms $\psi$ and $\varphi$ satisfy
\begin{align}
\partial_t \psi &= -\frac{7}{4}e^{4\phi}(\mathrm{d}\tau_0 + 4\tau_0 \mathrm{d}\phi)\wedge\varphi
+ e^{4\phi}\left(-\mathrm{Ric}_g - 2\mathcal{L}_{\mathrm{d}\phi} g + \frac{1}{2}E\right)\diamond\psi,\\
\partial_t \varphi &= \frac{7}{4}e^{4\phi}(\mathrm{d}\tau_0 + 4\tau_0 \mathrm{d}\phi)\lrcorner\psi
+ e^{4\phi}\left(-\mathrm{Ric}_g - 2\mathcal{L}_{\mathrm{d}\phi} g + \frac{1}{2}E\right)\diamond\varphi.
\end{align}
We first record two elementary observations. Let $\alpha$ be a tensor satisfying
\[
\partial_t \alpha = e^{4\phi}\beta
\]
for some tensor $\beta$. Then $\alpha-\tilde{\alpha}$ evolves by
\begin{align}\label{4.14}
    \partial_{t}(\alpha-\tilde{\alpha})&=e^{4\phi}\beta-e^{4\tilde{\phi}}\beta+e^{4\tilde{\phi}}\beta-e^{4\tilde{\phi}}\tilde{\beta}\\
    &=F\ast \beta+e^{4\tilde{\phi}}(\beta-\tilde{\beta})\notag,
\end{align}
for $|\alpha|^{2}$, by the assumptions, we also have
\begin{align}\label{4.15}
    \partial_{t}|\alpha|^{2}\leq \partial_{t} g (\alpha,\alpha)+2\langle \alpha,\partial_{t}\alpha\rangle\leq C|\alpha|^{2}+2\langle \alpha,\partial_{t}\alpha\rangle.
\end{align}
For $\mathcal{L}_{\mathrm{d} \phi} g$, by the assumptions, we have
\begin{align}\label{4.16}
    \left|\mathcal{L}_{\mathrm{d} \phi} g-\mathcal{L}_{\mathrm{d} \tilde{\phi}} \tilde{g}\right|&=2\left|\nabla^{2}\phi-\tilde{\nabla}^{2}\tilde{\phi}\right|\\
    &=\left|\nabla\mathbf{T}\ast\varphi+\mathbf{T}\ast\mathbf{T}\ast\psi-\tilde{\nabla}\tilde{\mathbf{T}}\ast\tilde{\varphi}-\tilde{\mathbf{T}}\ast\tilde{\mathbf{T}}\ast\tilde{\psi}\right|\notag\\
    &=\left|\nabla\mathbf{T}\ast Q+V\ast\tilde{\varphi}+U\ast\mathbf{T}\ast\psi+U\ast\tilde{\mathbf{T}}\ast\psi+\tilde{\mathbf{T}}\ast\tilde{\mathbf{T}}\ast P\right|\notag\\
    &\leq C|Q|+C|V|+C|U|+C|P|\notag.
\end{align}
For $E$, we also have
\begin{align}\label{4.17}
    \left|E-\tilde{E}\right|&=\left|\mathbf{T}\ast\mathbf{T}\ast g\ast(1+\varphi+\psi)-\tilde{\mathbf{T}}\ast\tilde{\mathbf{T}}\ast \tilde{g}\ast(1+\tilde{\varphi}+\tilde{\psi})\right|\\
    &=\left|U\ast\mathbf{T}\ast g\ast(1+\varphi+\psi)+\tilde{\mathbf{T}}\ast U\ast g\ast(1+\varphi+\psi)\right.\notag\\
    &\quad\left.+\tilde{\mathbf{T}}\ast\tilde{\mathbf{T}}\ast h\ast(1+\tilde{\varphi}+\tilde{\psi})+\tilde{\mathbf{T}}\ast\tilde{\mathbf{T}}\ast \tilde{g}\ast(1+Q+P)\right|\notag\\
    &\leq C|h|+C|U|+C|P|+C|Q|\notag.
\end{align}
{\bf Step 1:} estimate $|h|^{2}$:

By \eqref{4.14}-\eqref{4.17}, the evolution equation of $h(t)$ satisfies
\begin{align}
    \frac{\partial}{\partial t}h(t)&=F\ast({\rm Ric}+ \nabla^{2}\phi+E)\notag\\
    &\quad+e^{4\tilde{\phi}}\left[-2({\rm Ric}-\widetilde{{\rm Ric}}) -4(\mathcal{L}_{\mathrm{d} \phi} g-\mathcal{L}_{\mathrm{d} \tilde{\phi}} \tilde{g})+(E-\tilde{E})\right]\notag.
\end{align}
Thus, using Cauchy-Schwarz inequality, the evolution equation of $|h|^{2}$ satisfies
\begin{align}
    \frac{\partial}{\partial t}|h|^{2}&\leq C|F|^{2}+C|S|^{2}+C|h|^{2}+C|U|^{2}+C|P|^{2}+C|Q|^{2}+C|V|^{2}\notag\\
    &\leq C|\Phi|^{2}\notag.
\end{align}
{\bf Step 2:} estimate $|A|^{2}$:

Recall the evolution equation of Christoffel symbol
\begin{align}
    \frac{\partial}{\partial t}\Gamma^k_{ij}=\frac{1}{2}g^{kl}(\nabla_i\partial_tg_{jl}+\nabla_j\partial_tg_{il}-\nabla_l\partial_tg_{ij}).\notag
\end{align}
Using this equation, we can derive the evolution equation of $A$
\begin{align}
    \frac{\partial}{\partial t}A&=g^{-1}\ast \nabla(\partial_{t}g)-\tilde{g}^{-1}\ast \tilde{\nabla}(\partial_{t}\tilde{g})\notag\\
    &=(g^{-1}-\tilde{g}^{-1})\ast\tilde{\nabla}(\partial_{t}{\tilde{g}})+{g}^{-1}\ast (\nabla(\partial_{t}{g})-\tilde{\nabla}(\partial_{t}{\tilde{g}}))\notag\\
    &= \tilde{g}^{-1}\ast h\ast\tilde{\nabla}(\partial_{t}\tilde{g})+{g}^{-1}\ast e^{4\phi}\ast\mathbf{T}\ast\varphi\ast({\rm Ric}+ \nabla^{2}\phi+E)\notag\\
    &\quad -{g}^{-1}\ast e^{4\tilde{\phi}}\ast\tilde{\mathbf{T}}\ast\tilde{\varphi}\ast (\widetilde{{\rm Ric}}+ \tilde{\nabla}^{2}\tilde{\phi}+\tilde{E})+{g}^{-1}\ast e^{4\phi}\ast\nabla({\rm Ric}+ \nabla^{2}\phi+E)\notag\\
    &\quad-{g}^{-1}\ast e^{4\tilde{\phi}}\ast\tilde{\nabla} (\widetilde{{\rm Ric}}+ \tilde{\nabla}^{2}\tilde{\phi}+\tilde{E})\notag.
\end{align}
It follows that
\begin{align}
    \left|\frac{\partial}{\partial t}A\right|&\leq C|F|+C|S|+C|h|+C|U|+C|P|+C|Q|+C|V|+C|A|\notag\\
    &\quad+Ce^{4\phi}|\nabla S|+Ce^{4\phi}|\nabla V|\notag.
\end{align}
Using Cauchy-Schwarz inequality, \eqref{4.15} and Young's inequality, we also obtained that
\begin{align}\label{A}
\frac{\partial}{\partial t}|A|^{2}&\leq C|F|^{2}+C|S|^{2}+C|h|^{2}+C|U|^{2}+C|P|^{2}+C|Q|^{2}\\
    &\quad+C|V|^{2}+C|A|^{2}+\frac{1}{10}e^{4\phi}|\nabla S|^{2}+\frac{1}{10}e^{4\phi}|\nabla V|^{2}\notag\\
     &\leq C|\Phi|^{2}+\frac{1}{10}e^{4\phi}|\nabla S|^{2}+\frac{1}{10}e^{4\phi}|\nabla V|^{2}\notag.
\end{align}
{\bf Step 3:} estimate $|F|^{2}$:

From the flow $\eqref{heterotic flow}$ and \eqref{4.14}-\eqref{4.17}, the evolution equation of $F$ satisfies
\begin{align}
    \frac{\partial}{\partial t}F&=F\ast e^{4\phi}\ast g^{-1}\ast(\nabla^{2}\phi+\mathbf{T}\ast\mathbf{T})\notag\\
    &\quad+e^{4\tilde{\phi}}\ast\left[e^{4\phi}\ast g^{-1}\ast(\nabla^{2}\phi+\mathbf{T}\ast\mathbf{T})-e^{4\tilde{\phi}}\ast\tilde{g}^{-1}\ast(\tilde{\nabla}^{2}\tilde{\phi}+\tilde{\mathbf{T}}\ast\tilde{\mathbf{T}})\right]\notag.
\end{align}
Thus, we have the following estimate
\begin{align}
     \left|\frac{\partial}{\partial t}F\right|&\leq C|F|+C|h|+C|U|+C|P|+C|Q|+C|V|\notag,
\end{align}
which implies 
\begin{align}
    \frac{\partial}{\partial t}|F|^{2}&\leq C|F|^{2}+C|h|^{2}+C|U|^{2}+C|P|^{2}+C|Q|^{2}+C|V|^{2}\notag\\
    &\leq C|\Phi|^{2}\notag.
\end{align}
{\bf Step 4:} estimate $|P|^{2}$:

Recall the evolution equation of $\psi$, similar to the above steps, we have
\begin{align}
    \frac{\partial}{\partial t}P&=F\ast \left[(\nabla\mathbf{T}+\mathbf{T}\ast\mathbf{T}\ast\varphi)\ast \varphi+({\rm Ric}+\mathcal{L}_{\mathrm{d} \phi} g+E)\ast\psi\right]\notag\\
    &\quad +e^{4\tilde{\phi}}\ast(\nabla\mathbf{T}+\mathbf{T}\ast\mathbf{T}\ast\varphi)\ast \varphi-e^{4\tilde{\phi}}\ast(\tilde{\nabla}\tilde{\mathbf{T}}+\tilde{\mathbf{T}}\ast\tilde{\mathbf{T}}\ast\tilde{\varphi})\ast \tilde{\varphi}\notag\\
    &\quad+e^{4\tilde{\phi}}\ast({\rm Ric}+\mathcal{L}_{\mathrm{d} \phi} g+E)\ast\psi-e^{4\tilde{\phi}}\ast(\widetilde{\rm Ric}+\mathcal{L}_{\mathrm{d} \tilde{\phi}} \tilde{g}+\tilde{E})\ast\tilde{\psi}\notag.
\end{align}
Thus, we have
\begin{align}
    \left|\frac{\partial}{\partial t}P\right|&\leq C|F|+C|h|+C|U|+C|P|+C|Q|+C|V|+C|S|\notag,
\end{align}
which means
\begin{align}\label{P}
    \frac{\partial}{\partial t}|P|^{2}&\leq C|F|^{2}+C|h|^{2}+C|U|^{2}+C|P|^{2}+C|Q|^{2}+C|V|^{2}+C|S|^{2}\\
    &\leq C|\Phi|^{2}\notag.
\end{align}
{\bf Step 5:} estimate $|Q|^{2}$:

The evolution equation of $\varphi$ is similar to $\psi$, we can directly obtain the estimates of the evolution equation of $|Q|^{2}$, which is
\begin{align}\label{Q}
    \frac{\partial}{\partial t}|Q|^{2}&\leq C|F|^{2}+C|h|^{2}+C|U|^{2}+C|P|^{2}+C|Q|^{2}+C|V|^{2}+C|S|^{2}\\
    &\leq C|\Phi|^{2}\notag.
\end{align}
{\bf Step 6:} estimate $|U|^{2}$:

From (5.28) in \cite{GMPS26} with the setting $C=-\frac{4}{3}$, we get
\begin{align}
    \frac{\partial}{\partial t}\mathbf{T}=e^{4\phi}\left[\Delta\mathbf{T}+\mathbf{T}\ast({\rm Rm}+\nabla\mathbf{T}+\mathbf{T}\ast\mathbf{T})\ast(1+\phi+\psi)\right]
\end{align}
Thus, from \eqref{4.14}-\eqref{4.17}, we have
\begin{align}
    \frac{\partial}{\partial t}U&=F\ast \left(\tilde{\Delta}\tilde{\mathbf{T}}+\tilde{\mathbf{T}}\ast(\widetilde{\rm Rm}+\tilde{\nabla}\tilde{\mathbf{T}}+\tilde{\mathbf{T}}\ast\tilde{\mathbf{T}})\ast(1+\tilde{\phi}+\tilde{\psi})\right)\notag\\
    &\quad +e^{4{\phi}}\ast(g^{-1}\ast\nabla^{2}\mathbf{T}-\tilde{g}^{-1}\ast\tilde{\nabla}^{2}\tilde{\mathbf{T}})\notag\\
    &\quad+e^{4{\phi}}\ast\mathbf{T}\ast({\rm Rm}+\nabla\mathbf{T}+\mathbf{T}\ast\mathbf{T})\ast(1+\phi+\psi)\notag\\
    &\quad-e^{4{\phi}}\ast\tilde{\mathbf{T}}\ast(\widetilde{{\rm Rm}}+\tilde{\nabla}\tilde{\mathbf{T}}+\tilde{\mathbf{T}}\ast\tilde{\mathbf{T}})\ast(1+\tilde{\phi}+\tilde{\psi}),\notag
\end{align}
so that
\begin{align}
    \left|\frac{\partial}{\partial t}U\right|&\leq  C|F|+C|h|+C|A|+C|U|+C|P|+C|Q|+C|V|+C|S|+Ce^{4\phi}|\nabla V|\notag.
\end{align}
Then, we get
\begin{align}
    \frac{\partial}{\partial t}|U|^{2}&\leq C|F|^{2}+C|h|^{2}+C|A|^{2}+C|U|^{2}+C|P|^{2}+C|Q|^{2}\notag\\
    &\quad+C|V|^{2}+C|S|^{2}+\frac{1}{10}e^{4\phi}|\nabla V|^{2}\notag\\
    &\leq C|\Phi|^{2}+\frac{1}{10}e^{4\phi}|\nabla V|^{2}\notag.
\end{align}
{\bf Step 7:} estimate $|S|^{2}$:

From (5.16) in \cite{GMPS26}, we can use "$\ast$" notation to write the evolution equation of $\mathrm{Rm}$
\begin{align}
    \frac{\partial}{\partial t}{\rm Rm}&=e^{4\phi}\ast(\Delta{\rm Rm}+{\rm Rm}\ast{\rm Rm}+\nabla{\rm Rm}\ast\mathbf{T}\ast\varphi+{\rm Rm}\ast\nabla\mathbf{T}\ast\varphi)\notag\\
    &\quad+e^{4\phi}\ast{\rm Rm}\ast \mathbf{T}\ast\mathbf{T}\ast\psi\notag\\
    &\quad+e^{4\phi}\ast(\nabla^{2}\mathbf{T}\ast \mathbf{T}+\nabla\mathbf{T}\ast \nabla\mathbf{T}+\nabla\mathbf{T}\ast \mathbf{T}\ast \mathbf{T}+\mathbf{T}\ast \mathbf{T}\ast\mathbf{T}\ast \mathbf{T})\notag\\
    &\quad \ast (1+\varphi+\psi)^{2}\notag
\end{align}
and 
\begin{align}\label{4.19}
    \Delta{\rm Rm}-\tilde{\Delta}\widetilde{\rm Rm}&=\nabla_ig^{ij}\nabla_j{\rm Rm}-\tilde\nabla_i\tilde g^{ij}\tilde\nabla_j\widetilde{\rm Rm}\\
    &=\nabla_ig^{ij}\nabla_j({\rm Rm}-\widetilde{\rm Rm})+\nabla_i(g^{ij}\nabla_j-\tilde g^{ij}\tilde{\nabla}_j)\widetilde{\rm Rm}\notag\\
    &\quad+(\nabla_i-\tilde{\nabla}_i)(\tilde g^{ij}\tilde\nabla_j{\rm Rm})\notag\\
    &=\Delta S+\nabla_i(g^{ij}\nabla_j\widetilde{\rm Rm}-\tilde g^{ij}\tilde\nabla_j\widetilde{\rm Rm})+A*\tilde{\nabla}\widetilde{\rm Rm}.\notag
\end{align}
Similar to the above steps, it follows that
\begin{align}
     \frac{\partial}{\partial t}S&=F\ast(\tilde{\Delta}\widetilde{{\rm Rm}}+\widetilde{{\rm Rm}}\ast \widetilde{{\rm Rm}}+\tilde{\nabla}\widetilde{{\rm Rm}}\ast\tilde{\mathbf{T}}\ast\tilde{\varphi}+\widetilde{{\rm Rm}}\ast\tilde{\nabla}\tilde{\mathbf{T}}\ast\tilde{\varphi})\notag\\
     &\quad +F\ast (\tilde{\nabla}^{2}\tilde{\mathbf{T}}\ast \tilde{\mathbf{T}}+\tilde{\nabla}\tilde{\mathbf{T}}\ast \tilde{\nabla}\tilde{\mathbf{T}}+\tilde{\nabla}\tilde{\mathbf{T}}\ast \tilde{\mathbf{T}}\ast \tilde{\mathbf{T}}+\tilde{\mathbf{T}}^{4})\ast(1+\tilde{\varphi}+\tilde{\psi})^{2}\notag\\
     &\quad  +F\ast\widetilde{{\rm Rm}}\ast \tilde{\mathbf{T}}\ast\tilde{\mathbf{T}}\ast\tilde{\psi}+e^{4{\phi}}\ast(\Delta{\rm Rm}+{\rm Rm}\ast{\rm Rm}-\tilde{\Delta}\widetilde{{\rm Rm}}-\widetilde{{\rm Rm}}\ast \widetilde{{\rm Rm}})\notag\\
     &\quad +e^{4{\phi}}\ast(\nabla{\rm Rm}\ast\mathbf{T}\ast\varphi+{\rm Rm}\ast\nabla\mathbf{T}\ast\varphi-\tilde{\nabla}\widetilde{{\rm Rm}}\ast\tilde{\mathbf{T}}\ast\tilde{\varphi}-\widetilde{{\rm Rm}}\ast\tilde{\nabla}\tilde{\mathbf{T}}\ast\tilde{\varphi})\notag\\
     &\quad +e^{4{\phi}}\ast({\rm Rm}\ast \mathbf{T}\ast\mathbf{T}\ast\psi-\widetilde{{\rm Rm}}\ast \tilde{\mathbf{T}}\ast\tilde{\mathbf{T}}\ast\tilde{\psi})\notag\\
     &\quad +e^{4{\phi}}\ast (\nabla^{2}\mathbf{T}\ast \mathbf{T}+\nabla\mathbf{T}\ast \nabla\mathbf{T}+\nabla\mathbf{T}\ast \mathbf{T}\ast \mathbf{T}+\mathbf{T}^{4})\ast(1+\varphi+\psi)^{2}\notag\\
     &\quad -e^{4{\phi}}\ast (\tilde{\nabla}^{2}\tilde{\mathbf{T}}\ast \tilde{\mathbf{T}}+\tilde{\nabla}\tilde{\mathbf{T}}\ast \tilde{\nabla}\tilde{\mathbf{T}}+\tilde{\nabla}\tilde{\mathbf{T}}\ast \tilde{\mathbf{T}}\ast \tilde{\mathbf{T}}+\tilde{\mathbf{T}}^{4})\ast(1+\tilde{\varphi}+\tilde{\psi})^{2}\notag
\end{align}
     which implies
\begin{align}
     \left|\frac{\partial}{\partial t}S\right|&\leq e^{4\phi}(\Delta S(t)+{\rm div}\mathfrak{S}(t)) +C|F|+C|h|+C|A|+C|U|+C|P|+C|Q|\notag\\
     &\quad+C|V|+C|S|+Ce^{4\phi}|\nabla V|+Ce^{4\phi}|\nabla S|\notag,
\end{align}
where $\mathfrak{S}$ is defined by 
$$\mathfrak{S}^{id}_{abc}:=(g^{ij}\nabla_j-\tilde{g}^{ij}\tilde{\nabla}_{j})\tilde R^d_{abc}.$$
Using Young's inequality, it follows that
\begin{align}
    \frac{\partial}{\partial t}|S|^{2}&\leq C|S|^{2}+2\langle S,\partial_{t}S\rangle\notag\\
    &\leq 2\langle S,e^{4\phi}(\Delta S(t)+{\rm div}\mathfrak{S}(t))\rangle+C|F|^{2}+C|h|^{2}+C|A|^{2}+C|U|^{2}\notag\\
    &\quad +C|P|^{2}+C|Q|^{2}+C|V|^{2}+C|S|^{2}+\frac{1}{10}e^{4\phi}|\nabla V|^{2}+\frac{1}{10}e^{4\phi}|\nabla S|^{2}\notag\\
    &\leq 2\langle S,e^{4\phi}(\Delta S(t)+{\rm div}\mathfrak{S}(t))\rangle+C|\Phi|^{2}+\frac{1}{10}e^{4\phi}|\nabla V|^{2}+\frac{1}{10}e^{4\phi}|\nabla S|^{2}\notag.
\end{align}
{\bf Step 8:} estimate $|V|^{2}$:

Recall the evolution equation of $\mathbf{T}$, then we have
\begin{align}
    \nabla(\partial_{t}\mathbf{T})&=e^{4\phi}\Delta\nabla \mathbf{T}+e^{4\phi}\ast(\nabla{\rm Rm}\ast\mathbf{T}+{\rm Rm}\ast\nabla\mathbf{T}+{\rm Rm}\ast\mathbf{T}\ast\mathbf{T})\ast(1+\varphi+\psi)\notag\\
    &\quad+e^{4\phi}\ast(\nabla^{2}\mathbf{T}\ast\mathbf{T}+\nabla\mathbf{T}\ast\nabla\mathbf{T}+\nabla\mathbf{T}\ast\mathbf{T}\ast\mathbf{T}+\mathbf{T}^{4})\ast(1+\varphi+\psi)\notag.
\end{align}
With the same discussion, we have
\begin{align}
    \frac{\partial}{\partial t}V&=\nabla(\partial_{t}\mathbf{T})+\mathbf{T}\ast\nabla\partial_{t}g-\tilde{\nabla}(\partial_{t}\tilde{\mathbf{T}})-\tilde{\mathbf{T}}\ast\tilde{\nabla}\partial_{t}\tilde{g}\notag\\
    &=F\ast(\tilde{\Delta}\tilde{\nabla} \tilde{\mathbf{T}}+(\tilde{\nabla}\widetilde{\rm Rm}\ast\tilde{\mathbf{T}}+\widetilde{\rm Rm}\ast\tilde{\nabla}\tilde{\mathbf{T}}+\widetilde{\rm Rm}\ast\tilde{\mathbf{T}}\ast\tilde{\mathbf{T}})\ast(1+\tilde{\varphi}+\tilde{\psi}))\notag\\
    &\quad +F\ast(\tilde{\nabla}^{2}\tilde{\mathbf{T}}\ast\tilde{\mathbf{T}}+\tilde{\nabla}\tilde{\mathbf{T}}\ast\tilde{\nabla}\tilde{\mathbf{T}}+\tilde{\nabla}\tilde{\mathbf{T}}\ast\tilde{\mathbf{T}}\ast\tilde{\mathbf{T}}+\tilde{\mathbf{T}}^{4})\ast(1+\tilde{\varphi}+\tilde{\psi})\notag\\
    &\quad +e^{4{\phi}}\ast(\Delta\nabla \mathbf{T}-\tilde{\Delta}\tilde{\nabla} \tilde{\mathbf{T}})+U\ast \nabla\partial_{t}g+\tilde{\mathbf{T}}\ast(\nabla\partial_{t}g-\tilde{\nabla}\partial_{t}{\tilde{g}})\notag\\
    &\quad +e^{4{\phi}}\ast(\nabla{\rm Rm}\ast\mathbf{T}+{\rm Rm}\ast\nabla\mathbf{T}+{\rm Rm}\ast\mathbf{T}\ast\mathbf{T})\ast(1+\varphi+\psi)\notag\\
    &\quad -e^{4{\phi}}\ast(\tilde{\nabla}\widetilde{\rm Rm}\ast\tilde{\mathbf{T}}+\widetilde{\rm Rm}\ast\tilde{\nabla}\tilde{\mathbf{T}}+\widetilde{\rm Rm}\ast\tilde{\mathbf{T}}\ast\tilde{\mathbf{T}})\ast(1+\tilde{\varphi}+\tilde{\psi})\notag\\
    &\quad +e^{4{\phi}}\ast(\nabla^{2}\mathbf{T}\ast\mathbf{T}+\nabla\mathbf{T}\ast\nabla\mathbf{T}+\nabla\mathbf{T}\ast\mathbf{T}\ast\mathbf{T}+\mathbf{T}^{4})\ast(1+\varphi+\psi)\notag\\
    &\quad -e^{4{\phi}}\ast(\tilde{\nabla}^{2}\tilde{\mathbf{T}}\ast\tilde{\mathbf{T}}+\tilde{\nabla}\tilde{\mathbf{T}}\ast\tilde{\nabla}\tilde{\mathbf{T}}+\tilde{\nabla}\tilde{\mathbf{T}}\ast\tilde{\mathbf{T}}\ast\tilde{\mathbf{T}}+\tilde{\mathbf{T}}^{4})\ast(1+\tilde{\varphi}+\tilde{\psi})\notag
\end{align}
    By calculating, we get
\begin{align}
    \left|\frac{\partial}{\partial t}V\right|&\leq e^{4\phi}(\Delta V(t)+{\rm div}\mathfrak{V}(t))+C|F|+C|h|+C|A|+C|U|+C|P|+C|Q|\notag\\
     &\quad+C|V|+C|S|+Ce^{4{\phi}}|\nabla V|+Ce^{4{\phi}}|\nabla S|\notag,
\end{align}
where $\mathfrak{V}$ is defined by 
$$\mathfrak{V}^{i}_{abc}:=(g^{ij}\nabla_j-\tilde{g}^{ij}\tilde{\nabla}_{j})\tilde \nabla_a \tilde{\mathbf{T}}_{bc}.$$
Hence we get
\begin{align}
    \frac{\partial}{\partial t}|V|^{2}&\leq C|V|^{2}+2\langle V,\partial_{t}V\rangle\notag\\
    &\leq 2\langle V,e^{4\phi}(\Delta V(t)+{\rm \rm div}\mathfrak{V}(t))\rangle+C|F|^{2}+C|h|^{2}+C|A|^{2}+C|U|^{2}\notag\\
    &\quad +C|P|^{2}+C|Q|^{2}+C|V|^{2}+C|S|^{2}+\frac{1}{10}e^{4{\phi}}|\nabla V|^{2}+\frac{1}{10}e^{4{\phi}}|\nabla S|^{2}\notag\\
    &\leq 2\langle V,e^{4\phi}(\Delta V(t)+{\rm div}\mathfrak{V}(t))\rangle+C|\Phi|^{2}+\frac{1}{10}e^{4{\phi}}|\nabla V|^{2}+\frac{1}{10}e^{4{\phi}}|\nabla S|^{2}\notag.
\end{align}

From the estimate in Step 1 to Step 8, we complete this lemma.
\end{proof}

\begin{theorem}
    Suppose $(\psi(t),\phi(t)),(\tilde{\psi}(t),\tilde{\phi}(t))$ are  two solutions to the heterotic $G_{2}$ flow $\eqref{heterotic flow}$ on a compact manifold $M$ for $t\in[0,\epsilon],\epsilon>0$. If 
    $$(\psi(t),\phi(t))=(\tilde{\psi}(t),\tilde{\phi}(t))$$
    for some $t\in[0,\epsilon]$, then $(\psi(s),\phi(s))=(\tilde{\psi}(s),\tilde{\phi}(s))$ for all $s\in[t,\epsilon]$.
\end{theorem}
\begin{proof}
    Under the curvature and torsion bounds (\ref{uniformbdd}) and Shi-type estimates, the evolution equations of the metric $g(t)$, volume form $\mathrm{dVol}_{g(t)}$ and conformal factor $\phi(t)$ imply
\begin{align}
    \left|\frac{\partial}{\partial t}g(t)\right|_{g(t)}\leq C,~ \left|\frac{\partial}{\partial t}{\rm Vol}_{g(t)}\right|_{g(t)}\leq C,~\left|\frac{\partial}{\partial t}\phi(t)\right|_{g(t)}\leq C.
\end{align}
Using integration by parts, we get
\begin{align}
    \frac{\rm d}{{\rm d} t}\mathcal{W}(t)&\leq C \mathcal{W}(t)+\int_Me^{-4\phi}\frac{\partial}{\partial t}|\Phi(t)|^{2}{\rm d}{\rm Vol}_{g(t)}\notag\\
    &\leq C\mathcal{W}(t)+\int_M|\nabla V|^{2}+|\nabla S|^{2}{\rm dVol}_{g(t)}\notag\\
    &\quad+\int_M2\left(\left\langle S(t),\Delta S(t)+{\rm div}\mathfrak{S}(t)\right\rangle+\left\langle V(t),\Delta V(t)+{\rm div}\mathfrak{V}(t)\right\rangle\right){\rm dVol}_{g(t)}\notag\\
    &\leq C\mathcal{W}(t)+\int_M|\nabla V|^{2}+|\nabla S|^{2}{\rm dVol}_{g(t)}-2\int_M|\nabla V|^{2}+|\nabla S|^{2}{\rm dVol}_{g(t)}\notag\\
    &\quad-\int_M2\left(\left\langle\nabla S(t),\mathfrak{S}(t)\right\rangle+\left\langle\nabla V(t),\mathfrak{V}(t)\right\rangle\right){\rm dVol}_{g(t)}\notag\\
    &\leq C\mathcal{W}(t)\notag,
\end{align}
Suppose that $\psi(s)=\tilde\psi(s)$ and $\phi(s)=\tilde{\phi}(s)$ 
for some $s\in[0,\epsilon]$. Then $\mathcal{W}(s)=0$. Hence, for 
$t\in[s,\epsilon]$, Gronwall's inequality yields
\begin{equation}
    \mathcal{W}(t)\leq e^{C(t-s)}\mathcal{W}(s)=0,
\end{equation}
which implies that $\psi(t)=\tilde\psi(t)$ and 
$\phi(t)=\tilde{\phi}(t)$ for all $t\in[s,\epsilon]$, as required.
\end{proof}

\subsection{Backward uniqueness}
To complete the proof of uniqueness, it remains to establish the 
backward uniqueness of the heterotic $G_2$ flow.
For this purpose, we invoke a general backward  uniqueness theorem \cite{Kot10} for 
time-dependent sections of vector bundles satisfying certain  differential inequalities. Since we only consider compact manifolds, we state \cite{Kot10} here in this setting.

We define the operators
$$
\square := e^{4\phi}g^{i j} \nabla_i \nabla_j=e^{4\phi}\Delta \quad \text { and } \quad L := \frac{\partial}{\partial t}-\square.
$$
\begin{theorem}\label{thm5.4}
  Let $\mathcal{X}$ and $\mathcal{Y}$ be finite direct sums of the bundles $T_l^k(M)$, and let $X \in C^{\infty}(\mathcal{X} \times[0, \epsilon])$, $Y \in C^{\infty}(\mathcal{Y} \times[0, \epsilon])$ be smooth families of sections. Assume that there exist positive constants $C, \alpha_1$, and $\alpha_2$ such that
\begin{align}
\left|\frac{\partial}{\partial t}g\right|_{g(t)}^2+\left|\nabla \frac{\partial}{\partial t}g\right|_{g(t)}^2 \leq C, \quad\left|\frac{\partial}{\partial t} (e^{4\phi}g^{-1})\right|_{g(t)}^2+|\nabla( e^{4\phi}g^{-1})|_{g(t)}^2 \leq C,\notag
\end{align}
and 
$$
\alpha_1 g^{i j}(x, t) \leq e^{4\phi}g^{i j}(x, t) \leq \alpha_2 g^{i j}(x, t), 
$$
on $M \times[0, \epsilon]$ and that the metrics $g(t)$ are complete and satisfy
$$
\operatorname{Ric}(g(t)) \geq-K g(t) ,
$$
for some $K \geq 0$. We assume the sections $X, Y$ obey the following growth condition
$$
|X(x, t)|_{g(t)}^2+|(\nabla X)(x, t)|_{g(t)}^2+|Y(x, t)|_{g(t)}^2 \leq A e^{a d_{g(t)}\left(x_0, x\right)},
$$
for some $a, A>0$ and a fixed $x_0 \in M$, as well as the system of differential inequalities
\begin{align}
|L X|_{g(t)}^2 & \leq C\left(|X|_{g(t)}^2+|\nabla X|_{g(t)}^2+|Y|_{g(t)}^2\right)\label{LX}\\
\left|\frac{\partial Y}{\partial t}\right|_{g(t)}^2 & \leq C\left(|X|_{g(t)}^2+|\nabla X|_{g(t)}^2+|Y|_{g(t)}^2\right)\label{Yt}
\end{align}
for some $C \geq 0$. If $X(\cdot, \epsilon) \equiv 0$ and $Y(\cdot, \epsilon) \equiv 0$, then $X \equiv 0$ and $Y \equiv 0$ on $M \times[0, \epsilon]$.
\end{theorem}

Suppose $\psi(s)=\tilde{\psi}(s)$ and $\phi(s)=\tilde{\phi}(s)$ for some $s\in[0,\epsilon]$. We let
\begin{align}
    X(t)&=U(t)\oplus V(t)\oplus D(t)\oplus S(t)\oplus G(t),\notag\\
    Y(t)&=P(t)\oplus Q(t)\oplus h(t)\oplus A(t)\oplus B(t)\oplus F(t),\notag
\end{align}
where $U(t),V(t),S(t),P(t),Q(t),h(t),A(t),F(t)$ are defined as before and 
\begin{align}
    B=\nabla A,\quad D=\nabla^2\mathbf{T}-\tilde\nabla^2\tilde{\mathbf{T}},\quad G=\nabla{\rm Rm}-\tilde\nabla\widetilde{{\rm Rm}}.\label{BGD}
\end{align}
Then 
\begin{align}
    X(t)&\in T_2(M)\oplus T_3(M)\oplus T_4(M)\oplus T^1_3(M)\oplus T^1_4(M),\label{X(t)}\\
    Y(t)&\in T_4(M)\oplus T_3(M)\oplus T_2(M)\oplus T^1_2(M)\oplus T^1_3(M)\oplus T_0(M)\label{Y(t)}
\end{align}
We show that $X(t)$ and $Y(t)$ defined above satisfy the system of differential inequalities \eqref{LX}-\eqref{Yt}. We start with the following lemma.
\begin{lemma}
    The quantities $U(t),V(t),S(t),P(t),Q(t),h(t),A(t),F(t),B(t),D(t)$,$G(t)$ defined above are uniformly bounded with respect to $g(t)$ on $M\times[0,\epsilon]$.
\end{lemma}
\begin{proof}
    We have argued that two metrics $g(t)$ and $\tilde g(t)$ are uniformly equivalent on $M\times[0,\epsilon]$. We immediately deduce that $|h(t)|_{g(t)}=|g(t)-\tilde{g}(t)|_{g(t)}$ is bounded. From \eqref{uniformbdd} and the uniform equivalence of $g(t)$ and $\tilde{g}(t)$, we further have
\begin{align*}
|U|_{g(t)}&=| \mathbf{T}- \tilde{\mathbf{T}}|_{g(t)},\quad|V|_{g(t)}=|\nabla \mathbf{T}-\tilde{\nabla} \tilde{\mathbf{T}}|_{g(t)},\quad|D|_{g(t)}=|\nabla^2 \mathbf{T}-\tilde{\nabla}^2 \tilde{\mathbf{T}}|_{g(t)}, \\
|F|_{g(t)}&=|\phi-\tilde{\phi}|_{g(t)} ,\quad|S|_{g(t)}=|{\rm Rm}-\widetilde{\rm Rm}|_{g(t)}, \quad|G|_{g(t)}=|\nabla{\rm Rm}-\tilde{\nabla} \widetilde{\rm Rm}|_{g(t)}
\end{align*}
are bounded on $M \times[0, \epsilon]$. Since $\psi(s)=\tilde\psi(s)$ for some $s\in[0,\epsilon]$, using $s-t\leq\epsilon$ and (\ref{P})-(\ref{Q}), we have
\begin{align}
    \left|P(t) \right|^2+ \left|Q(t) \right|^2&=\left|P(t) \right|^2-\left|P(s) \right|^2+ \left|Q(t) \right|^2- \left|Q(s) \right|^2\notag\\
    &=\int_s^t\frac{\partial}{\partial \tau}(\left|P(\tau) \right|^2_{g(\tau)}+ \left|Q(\tau) \right|^2_{g(\tau)}){\rm d}\tau\notag\\
    &\leq C\int_s^t (|F(\tau)|^{2}+|h(\tau)|^{2}+|U(\tau)|^{2}+|P(\tau)|^{2}\notag\\
    &\qquad\quad+|Q(\tau)|^{2}+|V(\tau)|^{2}+|S(\tau)|^{2}){\rm d}\tau,\notag\\
    &\leq C\cdot\epsilon+C\int_t^s(|Q(\tau)|^2 + |P(\tau)|^2){\rm d}\tau.\notag
\end{align}
Using Gronwall's inequality, we have
\begin{equation}
    \left|P(t) \right|^2_{g(t)}+ \left|Q(t) \right|^2_{g(t)}\leq C\cdot\epsilon e^{-C(s-t)}\leq C\cdot \epsilon.
\end{equation}
Therefore $\left|P(t) \right|_{g(t)},\left|Q(t) \right|_{g(t)}$ are all bounded.

For $|A|^{2}$, recall that
\begin{align*}
\frac{\partial}{\partial t}|A|^{2}&\leq C|F|^{2}+C|S|^{2}+C|h|^{2}+C|U|^{2}+C|P|^{2}+C|Q|^{2}\notag\\
    &\quad+C|V|^{2}+C|A|^{2}+C|\nabla S|^{2}+C|\nabla V|^{2}.\notag
\end{align*}
For $|\nabla S|$ and $|\nabla V|$, we have the following estimates
\begin{align}
    |\nabla S|&=|\nabla{\rm Rm}-\nabla\widetilde{\rm Rm}|\notag\\
    &\leq |\nabla{\rm Rm}-\tilde{\nabla}\widetilde{\rm Rm}|+|\tilde{\nabla}\widetilde{\rm Rm}-\nabla\widetilde{\rm Rm}|\notag\\
    &\leq C|G|+C|A|,\notag
\end{align}
and
\begin{align}
    |\nabla V|&=|\nabla^2{\mathbf{T}}-\nabla\tilde{\nabla}\tilde{\mathbf{T}}|\notag\\
    &\leq |\nabla^2\mathbf{T}-\tilde{\nabla}^2\tilde{\mathbf{T}}|+|\tilde{\nabla}^2\tilde{\mathbf{T}}-\nabla\tilde{\nabla}\tilde{\mathbf{T}}|\notag\\
    &\leq C|D|+C|A|.\notag
\end{align}
Since $A(s)=0$, which yields
\begin{align}
    \left|A(t)\right|^2&=\left|A(t)\right|^2-\left|A(s)\right|^2=\int_s^t\frac{\partial}{\partial \tau}\left|A(\tau)\right|^2{\rm d}\tau\notag\\
    &\leq C\cdot\epsilon+C\int_t^s\left|A(\tau)\right|^2{\rm d}\tau\notag.
\end{align}
Using Gronwall's inequality again, we have
\begin{equation}
    \left|A(t) \right|^2_{g(t)}\leq C\cdot\epsilon e^{-C(s-t)}\leq C\cdot \epsilon. 
\end{equation}
Before we estimate $B$, we calculate the following
\begin{align*}
    \nabla\partial_{t}A&=g^{-1}\ast \nabla^{2}(\partial_{t}g)-\tilde{g}^{-1}\ast \nabla\tilde{\nabla}(\partial_{t}\tilde{g})-\nabla\tilde{g}^{-1}\ast \tilde{\nabla}(\partial_{t}\tilde{g})\\
    &=(g^{-1}-\tilde{g}^{-1})\ast \nabla^{2}(\partial_{t}g)+\tilde{g}^{-1}\ast (\nabla \nabla (\partial_{t}g)-\tilde{\nabla}\tilde{\nabla}(\partial_{t}\tilde{g}))\notag\\
    &\quad +A\ast\tilde{g}^{-1}\ast \tilde{\nabla}(\partial_{t}\tilde{g})\notag\\
    &= \tilde{g}^{-1}\ast h\ast \nabla^{2}(\partial_{t}g)+A\ast\tilde{g}^{-1}\ast \tilde{\nabla}(\partial_{t}\tilde{g})\notag\\
    &\quad +{\tilde{g}}^{-1}\ast \nabla(e^{4\phi}\ast\mathbf{T}\ast\varphi\ast({\rm Ric}+ \nabla^{2}\phi+E))\notag\\
    &\quad -{\tilde{g}}^{-1}\ast \tilde{\nabla}(e^{4\tilde{\phi}}\ast\tilde{\mathbf{T}}\ast\tilde{\varphi}\ast (\widetilde{{\rm Ric}}+ \tilde{\nabla}^{2}\tilde{\phi}+\tilde{E}))\notag\\
    &\quad +{\tilde{g}}^{-1}\ast\nabla( e^{4\phi}\ast\nabla({\rm Ric}+ \nabla^{2}\phi+E))\notag\\
    &\quad-{\tilde{g}}^{-1}\ast \tilde{\nabla}(e^{4\tilde{\phi}}\ast\tilde{\nabla} (\widetilde{{\rm Ric}}+ \tilde{\nabla}^{2}\tilde{\phi}+\tilde{E})).\notag
\end{align*}
Therefore we obtain the following estimate
\begin{align}\label{nabla A}
    |\nabla\partial_{t}A|&\leq C|F|+C|S|+C|h|+C|U|+C|P|+C|Q|+C|V|+C|A|\\
    &\quad+C|D|+C|G|+C|\nabla D|+C|\nabla G|\notag\\
    &\leq C|\Phi|+C|D|+C|G|+C|\nabla D|+C|\nabla G|.\notag
\end{align}
Since the norm of $\partial_t g$ is bounded, we have
\begin{align}
    |A\ast \partial_{t}g|\leq C|A|\notag.
\end{align}
Then we get $|B|$ satisfies
\begin{align}
    \left|B(t)\right|&=\left|\nabla A(t)\right|=\left|\nabla A(t)-\nabla A(s)\right|=\left|\int_s^t\frac{\partial}{\partial \tau}\nabla A(\tau){\rm d }\tau\right|\notag\\
    &\leq\int_t^s\left(\left| \nabla \partial_{\tau}A(\tau)\right|+\left|A(\tau)\ast \nabla\partial_{\tau}g(\tau)\right|\right){\rm d}\tau\notag\\
    &\leq C \int_t^s  |\Phi(\tau)|+|D(\tau)|+|G(\tau)|+|\nabla D(\tau)|+|\nabla G(\tau)|{\rm d}\tau\notag.
\end{align}
We observe that $\nabla G$ and $\nabla D$ 
\begin{align}
    |\nabla G|&=\left|\nabla^2{\rm Rm}-A\ast\tilde\nabla{\widetilde{\rm Rm}}-\tilde\nabla^2\widetilde{\rm Rm}\right|\leq C\left(\left|\nabla^2{\rm Rm}\right|+\left|\tilde\nabla^2{\widetilde{\rm Rm}}\right|+\left|A(t)\right|\right)\notag\\
    |\nabla D|&=\left|\nabla^3\mathbf{T}-A\ast\tilde\nabla^2\tilde{\mathbf{T}}-\tilde\nabla^3\tilde{\mathbf{T}}\right|\leq C\left(\left|\nabla^3\mathbf{T}\right|+\left|\tilde\nabla^3\tilde{\mathbf{T}}\right|+\left|A(t)\right|\right)\notag
\end{align}
are bounded. 
Therefore $\left|B(t)\right|_{g(t)}\leq C\cdot\epsilon$ since $s-t\leq\epsilon$.
\end{proof}
Now we estimate the evolution equation of $X$ and $Y$.
\begin{lemma}
    We have the following estimates on the evolution of $X$ and $Y$
    \begin{align}
|L X|_{g(t)}^2 & \leq C\left(|X|_{g(t)}^2+|\nabla X|_{g(t)}^2+|Y|_{g(t)}^2\right),\notag\\
\left|\frac{\partial Y}{\partial t}\right|_{g(t)}^2 & \leq C\left(|X|_{g(t)}^2+|\nabla X|_{g(t)}^2+|Y|_{g(t)}^2\right),\notag
\end{align}
where $C$ is a uniform constant.
\end{lemma}
\begin{proof}
From Lemma \ref{lemma5.2}, the terms $h, F,A,P,Q,V,S$ satisfy this Lemma, then we only need to calculate the terms contained $B,U,D,G$. 

{\bf Step 1:} estimate $B$.

From \eqref{nabla A}, we get
\begin{align}
    \left|\frac{\partial}{\partial t}B\right|&=\left|\frac{\partial}{\partial t}\nabla A\right|\leq\left|\nabla\frac{\partial}{\partial t}A\right|+\left|A*\nabla\frac{\partial}{\partial t}g\right|\notag\\
    &\leq C|\Phi|+C|D|+C|G|+C|\nabla D|+C|\nabla G|\notag\\
    &\leq C\left(|X|+|\nabla X|+|Y|\right)\notag.
\end{align}

{\bf Step 2:} estimate $U$.

For the evolution equation of $U$, we need to recalculate it here.
\begin{align}
    \frac{\partial}{\partial t}U&=F\ast \left(\tilde{\Delta}\tilde{\mathbf{T}}+\tilde{\mathbf{T}}\ast(\widetilde{\rm Rm}+\tilde{\nabla}\tilde{\mathbf{T}}+\tilde{\mathbf{T}}\ast\tilde{\mathbf{T}})\ast(1+\tilde{\phi}+\tilde{\psi})\right)\notag\\
    &\quad +e^{4{\phi}}(\Delta\mathbf{T}-\tilde{\Delta}\tilde{\mathbf{T}})+e^{4{\phi}}\ast\mathbf{T}\ast({\rm Rm}+\nabla\mathbf{T}+\mathbf{T}\ast\mathbf{T})\ast(1+\phi+\psi)\notag\notag\\
    &\quad-e^{4{\phi}}\ast\tilde{\mathbf{T}}\ast(\widetilde{{\rm Rm}}+\tilde{\nabla}\tilde{\mathbf{T}}+\tilde{\mathbf{T}}\ast\tilde{\mathbf{T}})\ast(1+\tilde{\phi}+\tilde{\psi})\notag.
\end{align}
From \eqref{4.19}, we have
\begin{align*}
    \left|LU\right|&\leq C|F|+C|h|+C|A|+C|U|+C|P|+C|Q|\notag\\
    &\quad+C|V|+C|S|+C|\nabla V|+C|B|\notag\\
    &\leq   C\left(|X|+|\nabla X|+|Y|\right).\notag
\end{align*}

{\bf Step 3:} estimate $D$.

We know that
\begin{align}
    D=\nabla^2\mathbf{T}-\tilde\nabla^2\tilde{\mathbf{T}}=\nabla\nabla\mathbf{T}-\nabla\tilde{\nabla}\tilde{\mathbf{T}}+\nabla\tilde{\nabla}\tilde{\mathbf{T}}-\tilde{\nabla}\tilde{\nabla}\tilde{\mathbf{T}}=\nabla V+A\ast \tilde{\nabla}\tilde{\mathbf{T}}\notag.
\end{align}
Thus, before we estimate $D$, we should calculate $\nabla(\partial_{t}V)$
\begin{align}
    \nabla(\partial_{t}V)&= \nabla (F\ast(\tilde{\Delta}\tilde{\nabla} \tilde{\mathbf{T}}+(\tilde{\nabla}\widetilde{\rm Rm}\ast\tilde{\mathbf{T}}+\widetilde{\rm Rm}\ast\tilde{\nabla}\tilde{\mathbf{T}}+\widetilde{\rm Rm}\ast\tilde{\mathbf{T}}\ast\tilde{\mathbf{T}})\ast(1+\tilde{\varphi}+\tilde{\psi})))\notag\\
    &\quad +\nabla (F\ast(\tilde{\nabla}^{2}\tilde{\mathbf{T}}\ast\tilde{\mathbf{T}}+\tilde{\nabla}\tilde{\mathbf{T}}\ast\tilde{\nabla}\tilde{\mathbf{T}}+\tilde{\nabla}\tilde{\mathbf{T}}\ast\tilde{\mathbf{T}}\ast\tilde{\mathbf{T}}+\tilde{\mathbf{T}}^{4})\ast(1+\tilde{\varphi}+\tilde{\psi}))\notag\\
    &\quad +\nabla (e^{4{\phi}}\ast(\Delta\nabla \mathbf{T}-\tilde{\Delta}\tilde{\nabla} \tilde{\mathbf{T}}))+\nabla (V\ast \partial_{t}g+\tilde{\nabla}\tilde{\mathbf{T}}\ast\partial_{t}{h})\notag\\
    &\quad +\nabla (e^{4{\phi}}\ast(\nabla{\rm Rm}\ast\mathbf{T}+{\rm Rm}\ast\nabla\mathbf{T}+{\rm Rm}\ast\mathbf{T}\ast\mathbf{T})\ast(1+\varphi+\psi))\notag\\
    &\quad -\nabla (e^{4{\phi}}\ast(\tilde{\nabla}\widetilde{\rm Rm}\ast\tilde{\mathbf{T}}+\widetilde{\rm Rm}\ast\tilde{\nabla}\tilde{\mathbf{T}}+\widetilde{\rm Rm}\ast\tilde{\mathbf{T}}\ast\tilde{\mathbf{T}})\ast(1+\tilde{\varphi}+\tilde{\psi}))\notag\\
    &\quad +\nabla (e^{4{\phi}}\ast(\nabla^{2}\mathbf{T}\ast\mathbf{T}+\nabla\mathbf{T}\ast\nabla\mathbf{T}+\nabla\mathbf{T}\ast\mathbf{T}\ast\mathbf{T}+\mathbf{T}^{4})\ast(1+\varphi+\psi))\notag\\
    &\quad -\nabla (e^{4{\phi}}\ast(\tilde{\nabla}^{2}\tilde{\mathbf{T}}\ast\tilde{\mathbf{T}}+\tilde{\nabla}\tilde{\mathbf{T}}\ast\tilde{\nabla}\tilde{\mathbf{T}}+\tilde{\nabla}\tilde{\mathbf{T}}\ast\tilde{\mathbf{T}}\ast\tilde{\mathbf{T}}+\tilde{\mathbf{T}}^{4})\ast(1+\tilde{\varphi}+\tilde{\psi}))\notag
\end{align}
Therefore we have
\begin{align}
    |\nabla(\partial_t V)-e^{4\phi}\Delta D|&\leq C|\Phi|+C|B|+C|D|+C|G|+C|\nabla D|+C|\nabla G|\notag\\
    &\leq C\left(|X|+|\nabla X|+|Y|\right)\notag.
\end{align}
which implies
\begin{align}
    |\partial_tD-e^{4\phi}\Delta D|&=|\nabla(\partial_{t}V)-e^{4\phi}\Delta D|+|V\ast\nabla\partial_{t} g|+|\partial_{t}A\ast\tilde{\nabla}\tilde{\mathbf{T}}|+|A\ast\partial_{t}\tilde{\nabla}\tilde{\mathbf{T}}|\notag\\
    &\leq C\left(|X|+|\nabla X|+|Y|\right)\notag.
\end{align}

{\bf Step 4:} estimate $G$.

With the same discussion in Step 3, we give the following estimate of $\nabla\partial_{t}S$
\begin{align}
    |\nabla\partial_{t}S-e^{4\phi}\Delta G| &\leq C|\Phi|+C|B|+C|D|+C|G|+C|\nabla D|+C|\nabla G|\notag\\
    &\leq C\left(|X|+|\nabla X|+|Y|\right)\notag.
\end{align}
Then, we get
\begin{align}
    |\partial_tG-e^{4\phi}\Delta G|&=|\nabla(\partial_{t}S)-e^{4\phi}\Delta G|+|S\ast\nabla\partial_{t} g|+|\partial_{t}A\ast\widetilde{\rm Rm}|+|A\ast\partial_{t}\widetilde{\rm Rm}|\notag\\
    &\leq C\left(|X|+|\nabla X|+|Y|\right)\notag.
\end{align}
Together with the above, we prove this lemma.
\end{proof}
\begin{theorem}
    Suppose $(\psi(t),\phi(t)),(\tilde{\psi}(t),\tilde{\phi}(t))$ are  two solutions to the heterotic $G_{2}$ flow $\eqref{heterotic flow}$ on a compact manifold $M$ for $t\in[0,\epsilon],\epsilon>0$. If 
    $$(\psi(t),\phi(t))=(\tilde{\psi}(t),\tilde{\phi}(t))$$
    for some $t\in[0,\epsilon]$, then $(\psi(s),\phi(s))=(\tilde{\psi}(s),\tilde{\phi}(s))$ for all $s\in[0,t]$.
\end{theorem}
\begin{proof}
    Since $M$ is compact and from the estimates \eqref{uniformbdd}, we have demonstrated that all of the conditions in Theorem \ref{thm5.4} are satisfied.
Hence, if $\psi(s)=\tilde{\psi}(s)$ and $\phi(s)=\tilde\phi(s)$ at some time $s \in[0, \epsilon]$, 
then $\psi(t)=\tilde\psi(t)$ and $\phi(t)=\tilde\phi(t)$ for all $t \in[0, s]$.
\end{proof}


\textbf{Acknowledgments.}\ \ 
 The first author is supported by China Postdoctoral Science Foundation (2026M793350). 
The second author is funded by Shanghai Institute for Mathematics and Interdisciplinary Sciences (SIMIS) under grant number SIMIS-ID-2025-AD.  The authors would also like to thank the referee for the valuable comments and suggestions.

\bibliographystyle{amsplain}

\end{document}